\documentclass{amsart}

\newtheorem{theorem}{Theorem}[section]
\newtheorem{lemma}[theorem]{Lemma}
\newtheorem{corollary}[theorem]{Corollary}
\usepackage{amssymb,amstext,amsthm,titletoc,fancyhdr,color,ulem}
\usepackage{amscd,amssymb,array,epsfig,booktabs}
\usepackage{graphicx}
\usepackage{epsf}
\usepackage{epstopdf}
\usepackage{ulem}
\usepackage{tikz}

\theoremstyle{definition}
\newtheorem{definition}[theorem]{Definition}
\newtheorem{example}[theorem]{Example}

\theoremstyle{remark}
\newtheorem{remark}[theorem]{Remark}

\numberwithin{equation}{section}

\begin{document}

\title[Difference Nevanlinna theories with vanishing and infinite periods] {Difference Nevanlinna theories with vanishing and infinite periods}
\author{Yik-Man Chiang}
\address{Department of Mathematics, Hong Kong University of Science and
Technology, Clear Water Bay, Kowloon, Hong Kong.}
\email{machiang@ust.hk}
\thanks{3rd March 2017}

\author{Xu-Dan Luo}
\address{Department of Mathematics, Hong Kong University of Science and
Technology, Clear Water Bay, Kowloon, Hong Kong.}
\curraddr{Department of Applied Mathematics, University of Colorado at Boulder, Boulder, Colorado 80309, USA.}
\email{lxdmathematics@gmail.com}
\thanks{This research was supported in part by the Research Grants Council
of the Hong Kong Special Administrative Region, China (16306315)}


\dedicatory{Dedicated to the memory of J. Milne Anderson.}


\begin{abstract}
By extending the idea of a difference operator with a fixed step to a varying-steps difference operator, we
  have established a difference Nevanlinna theory for meromorphic functions with the steps tending to zero (vanishing period) and a difference Nevanlinna theory for finite order meromorphic functions with the steps tending to infinity (infinite period) in this paper. We can recover the classical little Picard theorem from the vanishing period theory, but we require additional finite order growth restriction for meromorphic functions from the infinite period theory.
 Then we give some applications of our theories to exhibit connections between discrete equations and and their continuous analogues.
\end{abstract}

\maketitle



\section{Introduction}

Halburd and Korhonen \cite{Halburd 3} established a new Picard-type theorem and Picard-values with respect to difference operator $\Delta f(z)=f(z+1)-f(z)$ for finite-order meromorphic functions defined on $\mathbb{C}$ versus the classical Picard theorem and Picard values. More specifically, their theory allows them to show that if there are three $a_j-$points $j=1,\, 2,\, 3$ in $\hat{\mathbb{C}}$  such that each pre-image $f^{-1}(a_j)$
is an infinite sequence consisting of points lying on a straight line on which any two consecutive points differ by a fixed difference $c$ (but is otherwise arbitrary), then the function must be a periodic function with period $c$. This result can be considered as a discrete version of the  classical little Picard theorem for finite order meromorphic functions.
A crucial tool of their theory follows from their difference-type Nevanlinna theory for finite-order meromorphic functions,  is the difference logarithmic derivative lemma (see also \cite{Chiang 1}), i.e.,
\begin{equation*}
m\left(r, \frac{f(z+c)}{f(z)}\right)=o(T(r,f)),
\end{equation*}
where $c$ is a fixed nonzero constant. Instead of a fixed $c$, we define $g(z,c):=f(z+c)$, where $(z,c)\in \mathbb{C}^{2}$. Then $f(z+c)$ is a meromorphic function in $\mathbb{C}^{2}$. Moreover, a difference operator with \textit{varying steps} is defined by
\begin{equation*}
\label{E:steps}
\begin{split}
\Delta f_{c}&:=f(z+c)-f(z)\\
&=g(z,c)-g(z,0), \ \ \ (z,c)\in\mathbb{C}^{2}.
\end{split}
\end{equation*}

The first author of the current paper and Ruijsenaars \cite{Chiang-Ruijsenaars_2006} showed that for a nonzero meromorphic function $f(z)$ and $c\in \mathbb{C}$,
\begin{equation*}
m(r,f(z+c))<\frac{R+2r}{R-2r}m(R,f)+\sum_{l=0}^{L}\frac{1}{2\pi}\int_{0}^{2\pi}\log \left|\frac{R^{2}-\overline{b}_{l}(re^{i\phi}+c)}{R(re^{i\phi}+c-b_{l})}\right|d\phi,
\end{equation*}
where $|c|<r$ and $b_{0}, \cdots, b_{L}$ are poles of $f(z)$ in $|z|<R$, which implies an uniform bound as follows:
\begin{equation*}
m(r,f(z+c))\leq 5m(3r,f)+\log 4\cdot n(3r, f)
\end{equation*}
whenever $|c|<r$.

This uniform bound still hold if we restrict $c$, for examples,  such that $0<|c|<\frac{1}{r}$ or $\sqrt{r}<|c|< r$ when $r>1$, which will lead to a vanishing steps and an infinite steps when $r$ is sufficiently large, i.e., $0<|c|<\frac{1}{r}$ and $\sqrt{r}<|c|< r$ will result in $c\rightarrow 0$ and $c\rightarrow\infty$ respectively when $r\rightarrow\infty$. It motivates us to establish the corresponding difference Nevanlinna theories.

We consider the cases with \textit{vanishing period} and \textit{infinite period}, that is, when $c\to 0$ and $c\to\infty$ respectively.
On the one hand, if we denote  $c=\eta$ when it tends to zero via a sequence $\eta_n\to 0$, then the period guaranteed by Halburd-Korhonen's theory for each $n$ would tend to zero in a formal manner. Thus the periodic function with a vanishing period, when suitably defined, would \textit{formally} reduce to a  constant. On the other hand, if we denote  $c=\omega$ when it tends to infinity via  a sequence $\omega_n\to\infty$, then similarly the period as asserted by Halburd-Korhonen's theory for each $n$ would become infinite, and the distance between any two consecutive points on each of the three pre-image infinite sequences  would become sparse and \textit{eventually} reduce to a single point at most in the limit formally. In both cases that have been described, one would \textit{formally} recover the original little Picard theorem (namely the inverse images of each of three Picard values, must be a finite set at most).

In this paper, we {rigorously} establish {that} the above formal considerations indeed hold under certain senses. The upshot is that we can recover the classical little  Picard theorem as $\eta\to 0$ \textit{without} the finite-order restriction and \textit{with} the finite-order restriction when $\omega\to\infty$. In fact, our argument for our vanishing period results is \textit{independent} of Halburd-Korhonen's theory, while we apply methods similar to our earlier works \cite{Chiang 1, Chiang 2} and Halburd-Korhonen's theory \cite{Halburd 3, Chen_2014} in the infinite period results. We remark that the above finite order restriction in the infinite period case is necessary as it is unlikely that such results would hold for general meromorphic functions. However, the rates at which $\eta\to 0$ and $\omega\to \infty$ in the vanishing periods and infinite periods consideration respectively, generally depend on the growth of $f$.
\medskip

\smallskip
Hitherto we shall use the notation $\eta$ for $c$ when we consider the vanishing period case, and use the notation $\omega$ for $c$ when we consider the infinite period case. Thus when the rates at which  $\eta\to 0$ and $\omega\to \infty$ are suitably chosen, and \textit{Picard exceptional values} suitably defined respectively, we have obtained:
	\begin{enumerate}
		\item when a meromorphic function $f$ has three Picard exceptional values with respect to a \textit{varying-steps difference operator with vanishing period}, then $f$ is a constant (Theorem \ref{T:vanishing-picard});
		\item when a finite-order meromorphic function $f$ with three Picard exceptional values with respect to a \textit{varying-steps difference operator of infinite period}, then $f$ is a constant (Theorem \ref{T:infinite-picard}).
	\end{enumerate}
\medskip

\noindent The case (1) above gives an alternative proof of the original little Picard theorem. The case (2) requires additional finite-order restriction.
\bigskip

Let $f(z)$ be a meromorphic function, $\eta\in\mathbb{C}$ be a variable, for each fixed $r:=|z|$.
We introduce the symbols $m_{\eta}(r, f(z+\eta))$, $N_{\eta}(r, f(z+\eta))$ and $T_{\eta}(r,f(z+\eta))$ instead of $m(r, f(z+\eta))$, $N(r, f(z+\eta))$ and $T(r,f(z+\eta))$ when we want to emphasis that they are also functions of $\eta$. But we still have $m_\eta(r,\, f(z+\eta))=m(r,\, f(z+\eta))$, $N_\eta(r,\, f(z+\eta))=N(r,\, f(z+\eta))$, etc.
Our main estimates are as follows.

Let $f$ be an arbitrary meromorphic function and $0<|\eta|< \alpha_{1}(r)$, where $r=|z|$,
	\begin{equation*}
			\alpha_{1}(r)=\min\big\{\log^{-\frac{1}{2}}r,1/\left(n(r+1)\right)^{2}\big\},\quad  n(r)=n(r,\,f)+n\left(r,1/f\right).
	\end{equation*}
Then we obtain \textit{for each fixed $r$,}
	\begin{equation*}
		m_{\eta}\Big(r,\, \frac{f(z+\eta)}{f(z)}\Big)=o(1),
	\end{equation*}
\smallskip
as $\eta\rightarrow 0$. If, in addition, that $f$ has no pole in $\overline{D}(0,h)\setminus \{0\}$ for some positive $h$ and $0< |\eta|< \alpha_{2}(r)$, where
	\begin{equation*}
		\alpha_{2}(r)=\min \Big\{r, \log^{-\frac12}r, h/2,\frac{1}{\sum_{0<|b_{\mu}|<r+\frac{1}{2}}{1}/{|b_{\mu}|}}\Big\},
	\end{equation*}
here $(b_{\mu})_{\mu\in N}$ is the sequence of poles of $f(z)$, then for each fixed $r$
	\begin{equation*}
		 N_{\eta}\big(r,f(z+\eta)\big)=N\big(r,\,f(z)\big)+\varepsilon_{1}(r),
	\end{equation*}
where $|\varepsilon_{1}(r)|\leq n\left(0,f(z)\right)\log r+3$.

Although the above results hold without the finite-order restriction, the upper bounds of $|\eta|$, i.e., $\alpha_{1}(r)$ and $\alpha_{2}(r)$, which tend to zero as $r\to\infty$, are related to the growth of $f$.

When $f$ has positive finite order $\sigma$ and $\omega$ is suitably restricted by $0<|\omega|<r^\beta$, $0<~\beta<~1$,
\ then we have
	\begin{equation*}
		m\Big(r,\, \frac{f(z+\omega)}{f(z)}\Big)=O(r^{\sigma-(1-\beta){(1-\varepsilon)}+\varepsilon}),
	\end{equation*}
 and
{	\begin{equation*}
		N\big(r,\, f(z+\omega)\big)= N\big(r,\, f(z)\big)+O(r^{\sigma-(1-\beta)+\varepsilon})
	\end{equation*}
when $\sigma\geq 1$,
\begin{equation*}
		N\big(r,\, f(z+\omega)\big)= N\big(r,\, f(z)\big)+O(r^{\beta})
	\end{equation*}
when $0<\sigma< 1$}
\smallskip
for all $r$ outside a set of finite logarithmic measure. We have also obtained corresponding estimates for meromorphic functions with finite logarithmic order.

Finally, we show a different kind of vanishing period result for finite order meromorphic function:
	\begin{equation*}
		\lim\limits_{r\rightarrow
\infty}\lim\limits_{\eta\rightarrow
0}m_{\eta}\Big(r,\frac{1}{\eta}\Big(\frac{f(z+\eta)}{f(z)}-1\Big)\Big)=O(\log r)
\end{equation*}
\smallskip

\noindent thus recovering Nevanlinna's original logarithmic derivative estimate for finite order functions via yet another approach independent of previous methods although we do not have an immediate application of this result.
\medskip

This paper is organised as follows. We state the main theorems in \S\ref{S:main} and \S\ref{S:main'}. We shall establish Nevanlinna theory for difference operator in terms of vanishing and infinite periods in \S\ref{S:Nevanlinna theory-1} and \S\ref{S:Nevanlinna theory-2} respectively. We recall some known results in \S\ref{S:preliminaries}. The proofs of main results are given in \S 7 to \S 12. We exhibit some applications of our results to obtain classical differential equation results from their difference counterparts in \S \ref{S:applications}. A re-formulation of logarithmic derivative lemma and its proof are given in \S \ref{S:reformulation}. We shall use Nevanlinna's notation freely throughout this paper. See \cite{Hayman, Yang} for their meanings.

\bigskip

\section{Main results for vanishing period}\label{S:main}
In this section, our main results are for fixed $r:=|z|$. In this sense, $m\big(r, \frac{f(z+\eta)}{f(z)}\big)$, $N(r, f(z+\eta))$ and $T(r, f(z+\eta))$ are functions of $\eta$.  We sometimes write $m_{\eta}\big(r, \frac{f(z+\eta)}{f(z)}\big)$, $N_{\eta}(r, f(z+\eta))$ and $T_{\eta}(r, f(z+\eta))$ when we want to emphasize the dependence on $\eta$.
\medskip

\begin{theorem}
\label{T:limit-discrete-quotient}
Let $f(z)$ be a meromorphic function in $\mathbb{C}$ and $r=|z|$ be fixed. We have
\begin{equation}\label{E:proximty-limit-00}
							\lim\limits_{\eta\rightarrow 0}m_{\eta}\Big(r,\frac{f(z+\eta)}{f(z)}\Big)+\lim\limits_{\eta\rightarrow 0}m_{\eta}\Big(r,\frac{f(z)}{f(z+\eta)}\Big)=0.
						\end{equation}
Moreover, if we further assume $0<|\eta|< \alpha_{1}(r)$, where
	\begin{equation}
	    \label{E:alpha_1}
	    \alpha_{1}(r)=\min\big\{\log^{-\frac{1}{2}}r,1/\left(n(r+1)\right)^{2}\big\},\quad  n(r)=n(r,\,f)+n\left(r,1/f\right).
	\end{equation}
Then
\begin{equation}\label{E:proximty-limit-0}
		\lim\limits_{r\rightarrow \infty}m_{\eta}\Big(r,\,\frac{f(z+\eta)}{f(z)}\Big)+
		\lim\limits_{r\rightarrow
							\infty}m_{\eta}\Big(r,\frac{f(z)}{f(z+\eta)}\Big)=0.
	\end{equation}
\end{theorem}

We deduce from the above theorem the following corollary.
\medskip

\begin{corollary}
\label{C:1}
Let $f(z)$ be a meromorphic function, $0<r=|z|$ is a fixed but is otherwise arbitrary. Then
\begin{equation}
\lim\limits_{\eta\rightarrow
0}m_{\eta}\big(r,f(z+\eta)\big)=m\big(r,f(z)\big).
\end{equation}
\end{corollary}
\smallskip

Our next result is about a relation between the counting function and its varying steps.
\smallskip

\begin{theorem}
\label{T:limit-counting}
Let $f(z)$ be a meromorphic function in $\mathbb{C}$, and for each fixed $r=|z|$ such that $0< |\eta|< \alpha_{2}(r)$, where
	\begin{equation}\label{E:alpha_2}
		\alpha_{2}(r)=\min \Big\{r,\log^{-\frac{1}{2}}r, h/2,\frac{1}{\sum_{0<|b_{\mu}|<r+\frac{1}{2}}{1}/{|b_{\mu}|}}\Big\},
	\end{equation}
here $(b_{\mu})_{\mu\in N}$ is the sequence of poles of $f(z)$, and $h \in (0,\, 1)$ such that $f(z)$ {has no poles} in $\overline{D}(0,\, h)\setminus \{0\}$. Then
~\\
\begin{equation}
\label{E:counting-1}
	N_{\eta}\big(r,f(z+\eta)\big)=N\big(r,\,f(z)\big)+\varepsilon_{1}(r),
\end{equation}
~\\
where $|\varepsilon_{1}(r)|\leq n\left(0,f(z)\right)\log r+3$.
\end{theorem}

Combining the above asymptotic relations, we obtain the following estimate for the Nevanlinna characteristic function.

\begin{theorem}
\label{T:limit-characteristic}
Let $f(z)$ be a meromorphic function in $\mathbb{C}$. Then for each fixed $r=|z|$, there exists $\beta(r)>0$  with $\lim\limits_{r\rightarrow\infty}\beta(r)=0$ such that
~\\
\begin{equation}\label{E:limit-characteristic}
T_{\eta}\big(r,\, f(z+\eta)\big)= T\big(r,\,f(z)\big)+\varepsilon(r)
\end{equation}
~\\
whenever $0<|\eta|<\beta(r)$, where $|\varepsilon(r)|\leq n\left(0,f(z)\right)\log r+4$.
\end{theorem}

\begin{remark} We may choose $\beta(r)=\min\{\alpha_1(r),\, \alpha_2(r)\}$ defined above.
\end{remark}

In order to understand these asymptotic relations, we will give the following remark.
\begin{remark}
 Miles  showed in \cite{Miles} that the deficiency of meromorphic functions may change when choosing different origin. But Chiang and Feng \cite{Chiang 1} showed, for a fixed $\eta$,  the asymptotic relations
~\\
\begin{equation}\label{E:N_f_eta}
N\big(r,\,f(z+\eta)\big)=N(r,\,f)+O(r^{\lambda-1+\epsilon})+O(\log r)
\end{equation}
and
\begin{equation}\label{E:T_f_eta}
    T\big(r,\, f(z+\eta)\big)=T(r,\,f)+O(r^{\sigma-1+\varepsilon})+O(\log r)
\end{equation}
~\\
for finite order meromorphic function $f(z)$, where $\lambda$ denotes the exponent of convergence of poles of $f(z)$. This implies that the deficiency does not change after shifting the origin if the difference between the order and lower order is less than unity. By applying the estimates (\ref{E:N_f_eta}) and (\ref{E:T_f_eta}), one can easily obtain an alternative proof of an earlier result of Valiron \cite{Valiron} that if a finite order meromorphic function with the difference between its order and lower order is less then unity, then the  deficiency at the origin, i.e.,  $\delta(0)$ is invariant against any finite shift. This result of Valiron no longer hold in general. See Miles \cite{Miles}.
However, Theorems \ref{T:limit-counting} and \ref{T:limit-characteristic} indicate that the deficiency remains the same by allowing the period to tend to zero without any restriction of order.
\end{remark}
\medskip

\section{Main results for infinite period}\label{S:main'}

We distinguish two cases of meromorphic functions, that are, those with finite positive order and those with zero order of growth in this section.  In the former, the varying steps $\omega$  is restricted by  $0<|\omega|<r^\beta$ where the constant $\beta$ depends on the growth order of $f$.  In the latter, we have $0<|\omega| <\log^{\frac{1}{2}} r$. The $\omega$ is otherwise free to vary within the given upper bounds. For example,  the $|\omega|$ can tend to zero or to infinity when $r\rightarrow\infty$. In the case when we choose $\omega$ to be constant, then the results for finite order meromorphic functions would essentially agree with the results in Chiang and Feng   \cite{Chiang 1}.  We shall stick to the standard notations $m\big(r,\,\frac{f(z+\omega)}{f(z)}\big)$, $N(r, f(z+\omega))$ and $T(r, f(z+\omega))$ with the understanding that the $\omega$ is free to vary with respect to an upper bound that may depend on $f$ in this section.

\begin{theorem}
\label{T:limit-discrete-quotient'}
Let $f(z)$ be a meromorphic function of finite order $\sigma$, $0<\beta<1$ and $0<|\omega|<r^\beta$. Then given  {$0<\varepsilon<(1-\beta)/(2-\beta)$}, we have
\smallskip
				\begin{equation}
				 m\Big(r,\,\frac{f(z+\omega)}{f(z)}\Big)+ m\Big(r,\,\frac{f(z)}{f(z+\omega)}\Big)=O\big(r^{\sigma-(1-\beta)(1-\varepsilon)+\varepsilon}\big).
				\end{equation}
\end{theorem}
\smallskip

We note that the above upper bound, as well as latter consideration in this section, remains valid even when $\omega\to 0$ or remains constant, $\omega=1$, say.

Similarly, we have asymptotic relations for the Nevanlinna counting function and characteristic function of infinite period.
~\\
\begin{theorem}
\label{T:limit-counting'}
Let $f(z)$ be a meromorphic function of finite order $\sigma=\sigma(f)$.
	\begin{enumerate}
		\item[(i)] If $\sigma\geq 1$, $0<\beta<1$, $0<|\omega|<r^\beta$,{then there exists an} $0<\varepsilon<\beta^\prime$, where $\beta^\prime=\min\left\{ (\sigma-1)(1-\beta)/\beta,\ 1-\beta \right\}$, and we have
\begin{equation}\label{E:counting-eta-1}
N\big(r,\,f(z+\omega)\big)=N(r,\, f)+O\big(r^{\sigma-(1-\beta)+\varepsilon}\big)
\end{equation}
\\
\noindent {holds}	outside a set of finite logarithmic measure.
\item[(ii)]  If $0<\sigma<1$, $0<\beta<\sigma$, $0<|\omega|<r^\beta$, {then}  we have\\
 \begin{equation}\label{E:counting-eta-2}
N\big(r,\,f(z+\omega)\big)=N(r,\, f)+O\left(r^{\beta}\right)
\end{equation}
\\
\noindent {holds}	outside a set of finite logarithmic measure.
		\item[(iii)] If $\sigma=0$, $0<|\omega|<\log^{\frac{1}{2}} r$ for $r>1$, $0<|\omega|<1$ for $r\leq 1$, {then} we have
    \begin{equation}\label{E:counting-eta-3}
		N\big(r,\, f(z+\omega)\big)=N\big(r,\, f\big)+O\big(\log r\big)
\end{equation}
\noindent {holds} outside a set of finite logarithmic measure.
	\end{enumerate}
\end{theorem}

\begin{corollary}
\label{C:3}
Let $f(z)$ be a meromorphic function of finite logarithmic order $\sigma_{\log}=\lim\sup_{r\rightarrow\infty}{\log^{+}T(r,\, f)}/{\log\log r}>1$ and $0<|\omega|~<\log^{\beta}r$ where $1<\beta< \sigma_{\log}$ . Then we have
\smallskip
    \begin{equation}
N\big(r,\, f(z+\omega)\big)=N(r,\, f)+O(\log^{\beta}r)
\end{equation}
\noindent {holds} outside a set of finite logarithmic measure.
\medskip
\end{corollary}

We deduce from the Theorems \ref{T:limit-discrete-quotient'}, \ref{T:limit-counting'} and Corollary \ref{C:3} the following theorem.
\medskip

\begin{theorem}
\label{T:limit-characteristic'}
Let $f(z)$ be a meromorphic function of finite order $\sigma=\sigma(f)$ and let $\varepsilon>0$ denotes a positive constant.
	\begin{enumerate}
		\item[(i)] If {$\sigma\geq 1$}, let $0<\beta<1$ and $0<|\omega|<r^\beta$ such that
$0<\varepsilon<\beta''$, where $\beta''=\min\{(\sigma-1)(1-\beta)/\beta, (1-\beta)/(2-\beta)\}$, then we have

{\begin{equation}
	T\big(r,\, f(z+\omega)\big)=T(r,\, f)+O\big(r^{\sigma-(1-\beta)(1-\varepsilon)+\varepsilon}\big)
\end{equation}}

\noindent  {holds} outside a set of finite logarithmic measure.

\item[(ii)] If $0< \sigma< 1$, $0<\beta<\sigma$ and $0<|\omega|<r^\beta$,
then
we have
{\begin{equation}
	T\big(r,\, f(z+\omega)\big)=T(r,\, f)+O(r^{\beta})
\end{equation}}

\noindent outside a set of finite logarithmic measure.
\smallskip
	\item[(iii)] Moreover, if $\sigma=0$, let $0<|\omega|<\log^{\frac{1}{2}} r$ for $r>1$, $0<|\omega|<1$ for $r\leq 1$, {then} we have
    \begin{equation}
		T\big(r,\, f(z+\omega)\big)=T(r,\, f)+O(\log r)
\end{equation}
outside a set of finite logarithmic measure.
	\end{enumerate}
\end{theorem}

Then we immediately have the following corollary.

\begin{corollary}
\label{C:4}
Let $f(z)$ be a meromorphic function of finite logarithmic order  $\sigma_{\log}>1$.
Suppose $0<|\omega|<\log^{\beta}r$ with $1<\beta< \sigma_{\log}$. Then we have
~\\
    \begin{equation}
T\big(r,\,f(z+\omega)\big)=T(r,\, f)+O(\log^{\beta}r)
\end{equation}
~\\
outside a set of finite logarithmic measure.
\end{corollary}

\bigskip

\section{Nevanlinna theory for difference operator with vanishing period}\label{S:Nevanlinna theory-1}

We assume that the step size $c=\eta$ in (\ref{E:steps}) to be non-zero and whose upper bound tends to zero as $z\to\infty$ throughout this section.
\smallskip

\begin{theorem}
\label{T:second main-1-1}
Let $f(z)$ be a meromorphic function such that $\Delta_{\eta}f\not\equiv 0$ for each $z$, and let $p\geq 2$ be a positive integer, $a_{1},\cdots, a_{p}$ be $p$ distinct points in $\mathbb{C}$. Then there exists $\delta(r)>0$ such that
\begin{equation}
\label{E:second main-1-2}
m(r,\, f)+\sum_{k=1}^{p}m\big(r,\, 1/(f-a_{k})\big)\leq 2\, T(r,\, f)-N_{\Delta_{\eta}}(r,\, f)+\gamma
\end{equation} whenever $0<|\eta|<\delta(r)$,
where $\gamma$ is a constant which depends on $a_{1},\cdots, a_{p}$ and $r$ but it
is independent of $z$, and where
\begin{equation}
N_{\Delta_{\eta}}(r,f):=2N(r,f)-N(r,\Delta_{\eta}f)+N\left(r,1/\Delta_{\eta}f\right).
\end{equation}
\end{theorem}

\begin{proof}
We note that we shall use use the notation $\gamma_1,\, \gamma_2, \cdots$ to denote some definite constants that each of them depends on $a_{1},\cdots, a_{p}$ and $r$ but is independent of $z$ in our proof below.
Set
\begin{equation*}
P(f)=\prod_{k=1}^{p}\left(f-a_{k}\right),
\end{equation*}
we have
\begin{equation*}
\begin{split}
\sum_{k=1}^{p}m\left(r,\,1/(f-a_{k})\right)&=\sum_{k=1}^{p}T\left(r,\, 1/(f-a_{k})\right)
-\sum_{k=1}^{p}N\left(r,\, 1/(f-a_{k})\right)\\
&=pT(r,f)-N\left(r,1/P(f)\right)+\gamma_1\\
&=T(r,P(f))-N\left(r,1/P(f)\right)+\gamma_2\\
&=m\left(r,1/P(f)\right)+\gamma_2,
\end{split}
\end{equation*}

We deduce from \eqref{E:proximty-limit-00} of Theorem \ref{T:limit-discrete-quotient}, that for each fixed $r >0$, there is a $\delta(r) >0$ such that
	\begin{equation*}
	m\Big(r,\frac{f(z+\eta)}{f(z)}\Big)\le 1
	\end{equation*}
whenever $0<|\eta|<\delta(r)$. Then
\begin{equation*}
m\Big(r,\frac{\Delta_{\eta}f}{f-a_{k}}\Big)\leq\gamma_3,
\end{equation*}
which implies that
\begin{equation*}
m\Big(r,\frac{\Delta_{\eta}f}{P(f)}\Big)\leq\gamma_4.
\end{equation*}

We deduce
\begin{equation*}
\begin{split}
&\sum_{k=1}^{p}m\Big(r,\frac{1}{f-a_{k}}\Big)\leq m\Big(r,\frac{1}{\Delta_{\eta}f}\Big)+\gamma
=T\left(r,\Delta_{\eta}f\right)-N\Big(r,\frac{1}{\Delta_{\eta}f}\Big)+\gamma_5\\
&=m\Big(r,\ f\cdot\frac{\Delta_{\eta}f}{f}\Big)+N\left(r,\Delta_{\eta}f\right)
-N\Big(r,\frac{1}{\Delta_{\eta}f}\Big)+\gamma_5\\
&\leq m(r,f)+N\left(r,\Delta_{\eta}f\right)
-N\Big(r,\frac{1}{\Delta_{\eta}f}\Big)+\gamma_6.
\end{split}
\end{equation*}
Hence, there exists $\delta(r)>0$ and a constant $\gamma$ such that
\begin{equation*}
m(r,f)+\sum_{j=1}^{p}m\Big(r,\frac{1}{f-a_{j}}\Big)\leq 2T(r,f)-N_{\Delta_{\eta_{z}}}(r,f)+\gamma
\end{equation*}
whenever $0<|\eta|<\delta(r)$, where
\begin{equation*}
N_{\Delta_{\eta}}(r,f):=2N(r,f)-N(r,\Delta_{\eta}f)+N\Big(r,\frac{1}{\Delta_{\eta}f}\Big).
\end{equation*}
\end{proof}

\begin{definition}
Let $f(z)$ be a meromorphic function, $a$ be a finite complex number, the notation
	\begin{enumerate}
		\item[(i)]  $n_{\Delta_{\eta}}\left(r,1/(f-a)\right)$ represents the number of common zeros of $f-a$ and $\Delta_{\eta}f$ in $\overline{D(0,\, r)}:=\{z:|z|\leq r\}$ (counting multiplicity), and we define the multiplicity to be the minimum of those of $f-a$ and $\Delta_{\eta}f$ for such points;
		\item[(ii)] $n_{\Delta_{\eta}}\left(r,\, f\right):=n_{\Delta_{\eta}}\left(r;\,0;\, 1/f\right)$, which stands for the number of common zeros of $1/f$ and $\Delta_{\eta}1/f$ in $\overline{D(0,\, r)}$ (counting multiplicity), the multiplicity is defined to be the minimum of those of $1/f$ and $\Delta_{\eta}1/f$ for such points.
	\end{enumerate}
\end{definition}
\begin{definition}
\label{D:counting function}
We define the \textit{varying-steps difference integrated counting function} of $f(z)$ to be
\begin{equation}
N_{\Delta_{\eta}}\Big(r,\,\frac{1}{f-a}\Big)
:=\int_{0}^{r}\frac{n_{\Delta_{\eta}}\Big(t,\, \displaystyle\frac{1}{f-a}\Big)-n_{\Delta_{\eta}}\Big(0,\, \displaystyle\frac{1}{f-a}\Big)}{t}\,
dt\,+\, n_{\Delta_{\eta}}\Big(0,\, \frac{1}{f-a}\Big)\,\log r,
\end{equation}
\begin{equation}
N_{\Delta_{\eta}}\left(r,\, f\right)
:=\int_{0}^{r}\frac{n_{\Delta_{\eta}}\left(t,\,f\right)-n_{\Delta_{\eta}}\left(0,\,f\right)}{t}
dt+n_{\Delta_{\eta}}\left(0,\, f\right)\,\log r.
\end{equation}
Besides,
\begin{equation}
\widetilde{N}_{\Delta_{\eta}}\left(r,\, 1/(f-a)\right):=N\left(r,\, 1/(f-a)\right)
-N_{\Delta_{\eta}}\left(r,\,1/(f-a)\right),
\end{equation}
\begin{equation}
\widetilde{N}_{\Delta_{\eta}}\left(r,\,f\right):=N\left(r,\,f\right)
-N_{\Delta_{\eta}}\left(r,\,f\right).
\end{equation}
\end{definition}

We have the following Second Main Theorem for varying-steps difference operator with vanishing period.
~\\
\begin{theorem}
\label{T:second main-2}
Let $f(z)$ be a meromorphic function such that $\Delta_{\eta}f\not\equiv 0$ in $D(0,\, r):=\{z:|z|\leq r\}$. Let $a_{1},\cdots, a_{p}$ be $p\ge 2$ distinct points in $\mathbb{C}$. Then, there exists $\delta'(r)>0$ such that
\begin{equation}
(p-1)\,T(r,\,f)\leq \widetilde{N}_{\Delta_{\eta}}\left(r,\,f\right)+\sum_{k=1}^{p}\widetilde{N}_{\Delta_{\eta}}\left(r,\,1/(f-a_{k})\right)+\widehat{\varepsilon}(r)
\end{equation}
whenever $0<|\eta|< \delta'(r)$, where $|\widehat{\varepsilon}(r)| \leq n\left(0,\, f(z)\right)\log r+\gamma$,
here, $\gamma$ is a constant which depends only on $a_{1},\cdots, a_{p}$ and $r$ but is independent of $z$, and $\lim\limits_{r\rightarrow
\infty}\delta'(r)=0$.
\end{theorem}

\begin{proof}We deduce from Theorem \ref{T:second main-1-1}, after adding $\sum_{i=1}^pN(r,1/(f-a_i))$ and applying Nevanlinna's first fundamental theorem, that
\begin{equation}\label{E:thm4.4-step-1}
(p-1)\, T(r,\, f)\leq \sum_{k=1}^{p}N\big(r,1/(f-a_{k})\big)+N\big(r,\, \Delta_{\eta}f\big) -N\big(r,\, 1/\Delta_{\eta}f\big)-N(r,\, f)+\gamma
\end{equation}
whenever $0<|\eta|<\delta(r)$.

According to Definition \ref{D:counting function}, we have
\begin{equation*}
\sum_{k=1}^{p}N\left(r,\, 1/(f-a_{k})\right)-\sum_{k=1}^{p}\widetilde{N}_{\Delta_{\eta}}\left(r,\, 1/(f-a_{k})\right)
=\sum_{k=1}^{p}N_{\Delta_{\eta}}\left(r,\, 1/(f-a_{k})\right)\leq N\left(r,\, 1/\Delta_{\eta}f\right),
\end{equation*}
hence
\begin{equation}\label{E:thm4.4-step-2}
\sum_{k=1}^{p}N\left(r,1/(f-a_{k})\right)-N\left(r,1/\Delta_{\eta}f\right)
\leq \sum_{k=1}^{p}\widetilde{N}_{\Delta_{\eta}}\left(r,1/(f-a_{k})\right).
\end{equation}
~\\
~\\
Moreover, if $z=z_{0}$ is a common pole of $f(z)$ and $f(z+\eta)$ with multiplicity $m_{1}$ and $m_{2}$ respectively, then we can write
\begin{equation*}
f(z)=\frac{g(z)}{(z-z_{0})^{m_{1}}},\ \ \ f(z+\eta)=\frac{h(z)}{(z-z_{0})^{m_{2}}},
\end{equation*}
where both $g(z)$ and $h(z)$ are analytic at $z=z_{0}$ and $g(z_{0})\neq 0$, \ $h(z_{0})\neq 0$.
Without loss of generality, we may assume that $m_{1}\geq m_{2}$. Thus,
\begin{equation*}
\Delta_{\eta}f:=f(z+\eta)-f(z)=\frac{(z-z_{0})^{m_{1}-m_{2}}h(z)-g(z)}{(z-z_{0})^{m_{1}}}
\end{equation*}
and
\begin{equation*}
\Delta_{\eta}\frac{1}{f}:=\frac{1}{f(z+\eta)}-\frac{1}{f(z)}=\frac{(z-z_{0})^{m_{2}}\left[g(z)-(z-z_{0})^{m_{1}-m_{2}}h(z)\right]}{h(z)g(z)}.
\end{equation*}
~\\
If $m_{1}>m_{2}$, then the multiplicity for the pole of $\Delta_{\eta}f$ and the zero of $\Delta_{\eta}\frac{1}{f}$ at  $z=z_{0}$ are $m_{1}$ and $m_{2}$ respectively,from which it follows that the minimum multiplicity of the zero of $\frac{1}{f}$ and $\Delta_{\eta}\frac{1}{f}$ at $z=z_{0}$ is $m_{2}$. If $m_{1}=m_{2}$, we can write $h(z)-g(z)=(z-z_{0})^{m}\cdot d(z)$, where $m$ is a nonnegative integer, $d(z)$ is analytic at $z=z_{0}$ and $d(z_{0})\neq 0$. Thus, the multiplicity for the pole of $\Delta_{\eta}f$ at $z=z_{0}$ is $m_{1}-m$ if $m_{1}\geq m$, and is $0$ if $m_{1}<m$, while similar consideration for the zero of $\Delta_{\eta}\frac{1}{f}$ at $z=z_{0}$ is $m_{1}+m$, which implies that the minimum multiplicity of $\frac{1}{f}$ and $\Delta_{\eta}\frac{1}{f}$ at $z=z_{0}$ is $m_{1}$.
Hence, we deduce
~\\
\begin{equation*}
N\left(r,\Delta_{\eta}f\right)+N_{\Delta_{\eta}}\left(r,f\right)
\leq N(r,f)+N(r,f(z+\eta))
\end{equation*}
~\\
and Theorem \ref{T:limit-counting} guarantees that we can find $0<\delta'(r)<\delta(r)$ such that
~\\
\begin{equation}\label{E:thm4.4-step-3}
N\left(r,\Delta_{\eta}f\right)-N(r,f)\leq \widetilde{N}_{\Delta_{\eta}}\big(r,\,f\big)
+\varepsilon_{1}(r)
\end{equation}
~\\
whenever $0<|\eta|< \delta'(r)$. Combining \eqref{E:thm4.4-step-1},  \eqref{E:thm4.4-step-2} and \eqref{E:thm4.4-step-3} we deduce
\begin{equation*}
(p-1)\,T(r,\, f)\leq \widetilde{N}_{\Delta_{\eta}}\left(r,\,f\right)+\sum_{k=1}^{p}\widetilde{N}_{\Delta_{\eta}}\big(r,\, 1/(f-a_{k})\big)+\widehat{\varepsilon}(r)
\end{equation*}
~\\
whenever $0<|\eta|< \delta'(r)$, where $|\widehat{\varepsilon}(r)| \leq n\left(0,\, f(z)\right)\log r+\gamma$.
\end{proof}

\subsection{Defect relation and little Picard's theorem for varying-steps difference operator}
We define the \textit{multiplicity index} and \textit{ramification index} for varying-steps difference operator with vanishing period to be
\begin{equation*}
\vartheta_{\Delta_{\eta}}(a,\, f):=\lim\inf_{r\rightarrow \infty}\frac{N_{\Delta_{\eta}}\big(r,\,1/(f-a)\big)}{T(r,\,f)}
\end{equation*}
and
\begin{equation*}
\Theta_{\Delta_{\eta}}(a,\, f)
:=1-\lim\sup_{r\rightarrow\infty}\frac{\widetilde{N}_{\Delta_{\eta}}\big(r,\,1/(f-a)\big)}{T(r,\, f)}.
\end{equation*}
~\\
Then the Theorem \ref{T:second main-2} implies the following corollary immediately.
~\\
\begin{corollary}
\label{T:defect}
Let $f(z)$ be a transcendental meromorphic function such that
~\\
$\Delta_{\eta}f\not\equiv 0$. Then
\begin{equation*}
\sum_{a\in \widehat{\mathbb{C}}}\left(\delta(a,f)+\vartheta_{\Delta_{\eta}}(a,f)\right)
\leq \sum_{a\in \widehat{\mathbb{C}}}\Theta_{\Delta_{\eta}}(a,f)\leq 2.
\end{equation*}

\end{corollary}
~\\
~\\
Next, we shall define Picard exceptional values for varying-steps difference operator with vanishing period.
~\\
\begin{definition}
We call $a\in\widehat{\mathbb{C}}$ is a \textit{Picard exceptional value for varying-steps difference operator with vanishing period} of $f(z)$
if {there is a sequence $\eta_n\to 0$ as $n\to\infty$ such that} $\widetilde{N}_{\Delta_{\eta_{n}}}\left(r,1/(f-a)\right)=O(1)$.
\end{definition}
~\\
We have the following Picard theorem for varying-steps difference operator with vanishing period.
\begin{theorem}\label{T:vanishing-picard}
Let $f(z)$ be a meromorphic function having three Picard exceptional values for varying-steps difference operator with vanishing period. Then $f(z)$ is a constant.
\end{theorem}

\begin{proof}
Without loss of generality, we may assume that the three exceptional values to be $0$, $1$ and $\infty$. According to Theorem \ref{T:second main-2}, we can find a $\delta'(r)>0$ such that
\begin{equation*}
T(r,f)\leq \widetilde{N}_{\Delta_{\eta}}\big(r,\,f\big)+\widetilde{N}_{\Delta_{\eta}}\big(r,\, 1/f\big) +\widetilde{N}_{\Delta_{\eta}}\big(r,\,1/(f-1)\big)+\widehat{\varepsilon}(r)
\end{equation*}
whenever $0<|\eta|< \delta'(r)$, where $|\widehat{\varepsilon}(r)| \leq n\left(0,\, f(z)\right)\log r+\gamma$, here $\gamma$ is a bounded constant.
~\\
~\\
If $f(z)$ is a transcendental meromorphic function, then
\begin{equation*}
\lim\limits_{r\rightarrow\infty} \frac{T(r,f)}{\log r}=\infty.
\end{equation*}
Thus, there exists $r_{0}>0$ such that
\begin{equation*}
	T(r,\, f)>2\, \widehat{\varepsilon}(r)
\end{equation*}
whenever $r\geq r_{0}$. For each $ r\geq r_{0}$, since $\lim\limits_{n\rightarrow\infty} \eta_{n}=0\ (\eta_{n}\neq 0)$, so there exists $N(r)>0$ such that
$0<|\eta_{n}|< \delta'(r)$ whenever $n>N(r)$.
Note that
\begin{equation*}
\widetilde{N}_{\Delta_{\eta_{n}}}\left(r,\, f\right)=\widetilde{N}_{\Delta_{\eta_{n}}}\left(r\, ,1/f\right)
=\widetilde{N}_{\Delta_{\eta_{n}}}\big(r,\,1/(f-1)\big)=0
\end{equation*}
whenever $n>N(r)$. Thus, $T(r,f)\leq \widehat{\varepsilon}(r)$, which is a contradiction.

Hence, $\Delta_{\eta_{n}}f\equiv0$ on $\{z:|z|\leq r\}$ whenever $n>N(r)$.

We claim that $f(z)$ is an entire function. {For otherwise, there exists $z_{1}$ such that $f(z_{1})=\infty$, which implies that $1/f(z_{1}+\eta_{n})=1/f(z_{1})=0$ whenever $n>N(r_{1})$, where $r_{1}\geq \max\{r_{0},|z_{1}|\}$.
Note that $\lim\limits_{n\rightarrow\infty} \eta_{n}=0\ (\eta_{n}\neq 0)$, then $z_{1}$ is a non-isolated zero, which is a contradiction.} So $f(z)$ must be an entire function.

Moreover, we have $f(\eta_{n})=f(0)$ whenever $n>N(r_{0})$. By the Identity Theorem, we deduce that $f(z)\equiv f(0)$ on $\mathbb{C}$, which is impossible according to the assumption $f(z)$ being transcendental.

Therefore, $f(z)$ is a rational function, which must reduce to a constant.

\end{proof}

\section{Nevanlinna theory for difference operator with infinite period}\label{S:Nevanlinna theory-2}
When considering analogous Picard-exceptional values for varying steps operators with infinite periods, the Second Main Theorem that we state next allows $\omega$ to vary within the upper bound $r^\beta$. However, we need to further restrict $\omega$ to $r^{\beta/4}<|\omega|<r^{\beta}$ if $f(z)$ is of finite positive order, to $\log^{\frac{1}{8}} r<|\omega|<\log^{\frac{1}{2}} r$ if $f(z)$ is order zero.

\begin{definition}
\label{D: difference operator'}
Let $f(z)$ be a meromorphic function of finite order $\sigma$ and $r=|z|$. If $\sigma>0$, then we define \textit{varying-steps difference operator} by $\Delta_{\omega} f:= f(z+\omega)-f(z)$, where {$0<|\omega|<r^{\beta}$}, $0<\beta<\min\{1,\sigma\}$. If $\sigma=0$, then we define $\Delta_\omega f:= f(z+\omega)-f(z)$, where $0<|\omega|<\log^{\frac{1}{2}} r$ for $r>1$, and $0<|\omega|<1$ for $r\leq 1$.
\end{definition}
\smallskip

Based on this definition, we have the following versions of second main theorems.
\begin{theorem}
\label{T:second main-1-1'}
Let $f(z)$ be a meromorphic function of finite order $\sigma$ such that $\Delta_{\omega}f\not\equiv 0$. Let $a_{1},\cdots, a_{p}$ be $p\geq 2$ distinct points in $\mathbb{C}$. Then
\begin{equation}
\label{E:second main-1-2'}
m(r,\,f)+\sum_{j=1}^{p}m\big(r,\,1/(f-a_{j})\big)\leq 2 T(r,\,f)-N_{\Delta_{\omega}}(r,\, f)+o\big(T(r,\,f)\big)
+O(\log r)
\end{equation}
outside a set of finite logarithmic measure,
where
\begin{equation}
N_{\Delta_{\omega}}(r,\, f):=2N(r,\,f)-N(r,\,\Delta_{\omega}f)+N\big(r,\,1/\Delta_{\omega}f\big).
\end{equation}
\end{theorem}
\smallskip

\begin{theorem}
\label{T:second main-2'}
Let $f(z)$ be a meromorphic function of finite order $\sigma$ such that $\Delta_{\omega}f\not\equiv 0$. Let $a_{1},\cdots, a_{p}$ be $p\geq 2$ distinct points in $\mathbb{C}$. Then
\begin{equation}
(p-1)\,T(r,\,f)\leq \widetilde{N}_{\Delta_{\omega}}\left(r,\,f\right)+\sum_{j=1}^{p}\widetilde{N}_{\Delta_{\omega}}\big(r,\, 1/(f-a_{j})\big)
+o\big(T(r,\,f)\big)+O(\log r)
\end{equation}
outside a set of finite logarithmic measure.
\end{theorem}
\smallskip

An analogue of Picard exceptional values is defined as follows.

\begin{definition}
Let $f(z)$ be a meromorphic function of finite order $\sigma$, we call $a\in\widehat{\mathbb{C}}$ to be a \textit{Picard exceptional value for a varying-steps difference operator with infinite period} if $\widetilde{N}_{\Delta_{\omega}}\big(r,\,1/(f-a)\big)=O(1)$:

	\begin{enumerate}
		\item[i)] if $\sigma>0$,
			\begin{enumerate}
			\item[(a)] $|\omega|=|z|^{\frac{\beta}{2}}$, $0<\beta<\min\{1,\sigma\}$, for $|z|>1$, and $|z+\omega|< |z|-|z|^{\frac{\beta}{4}}$,
				\item[(b)] $|\omega|={|z|}/{2}$ for $|z|\leq 1$, and $|z+\omega|< \frac{3}{4}|z|$;
			\end{enumerate}	
		\item[ii)] if  $\sigma=0$,
			\begin{enumerate}
				\item[(a)]  $|\omega|=\log^{\frac{1}{4}}|z|$ for $|z|>1$, and $|z+\omega|< |z|-\log^{\frac{1}{8}}|z|$,
				\item[(b)] $|\omega|={|z|}/{2}$ for $|z|\leq 1$, and $|z+\omega|< \frac{3}{4}|z|$.
			\end{enumerate}
		\end{enumerate}
\end{definition}
\smallskip
\begin{center}
\begin{tikzpicture}[scale=0.5]
\draw[->] (0,0)--(12,0);
\draw[->] (6,-6)--(6,6);
\draw[style=dashed] (6,0) circle (5.0);
\draw[fill=black] (11,5.5) circle (0.05) node[above right] {$z_{0}$};
\path[draw] (6,0)--(10,3) node[above right]{$r_{0}=|z_{0}|-|z_{0}|^{\frac{\beta}{4}}$};
\draw[fill=black] (8,3.5) circle (0.05) node[above left] {$z_{1}:=z_{0}+\omega_{0}$};
\draw[style=dashed] (6,0) circle (3.0);
\path[draw] (6,0)--(4,2.236) node[above left]{$r_{1}=|z_{1}|-|z_{1}|^{\frac{\beta}{4}}$};
\draw[fill=black] (5,2.3) circle (0.05) node[below right] {$z_{2}:=z_{1}+\omega_{1}$};
\end{tikzpicture}
\begin{figure}[htp]
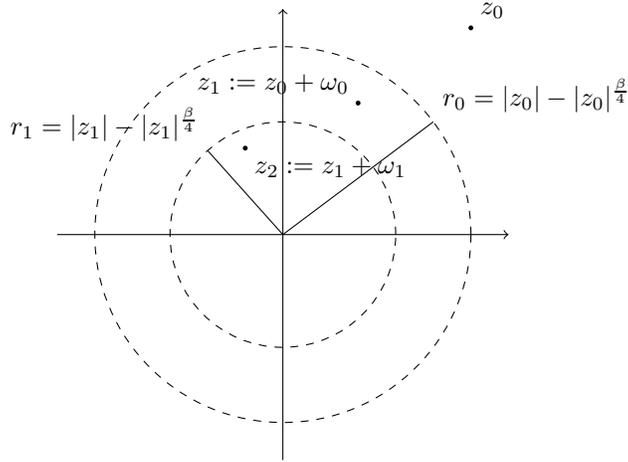

\caption{This figure shows the locations of three successive points which lie on the preimage of a \textit{Picard exceptional value for a varying-steps difference operator with infinite period}, where $z_{n}\rightarrow 0$ as $n\rightarrow \infty$.}
\end{figure}
\end{center}
Then we have the following Picard theorem for difference operator with infinite period.
\medskip

\begin{theorem}\label{T:infinite-picard}

Let $f(z)$ be a meromorphic function of finite order $\sigma$. Suppose $f$  has three Picard exceptional values with respect to a varying-steps difference operator with infinite period. Then $f(z)$ is a constant.
\end{theorem}
\medskip

\begin{proof}
Without loss of generality, we may assume that the three exceptional values to be  $0$, $1$ and $\infty$.
We deduce from Theorem \ref{T:second main-2'} that, if $\Delta_{\omega}f\not\equiv 0$, we have
\begin{equation*}
	T(r,\,f)\leq o(T(r,\, f))+O(\log r).
\end{equation*}
This is a contradiction unless either $f(z+\omega)\equiv f(z)$ and $f(z)$ is a transcendental meromorphic function or $f(z)$ is a rational function. We first show that if it is the latter, then $f$ must must reduce to a constant. For otherwise, the definition of Picard exceptional values for varying-steps difference operator with infinite period implies that the $f$ has an infinite sequence of zeros/poles/a-points, which is a contradiction.

Next, we consider the case of $f(z)$ being a transcendental meromorphic function with $f(z+\omega)\equiv f(z)$. We claim that $f(z)$ must be an entire function. Without loss of generality, we may assume that $\arg \omega=-\arg z$ which is guaranteed under the assumption $|z+\omega|< |z|-|z|^{\frac{\beta}{4}}$ and $|z+\omega|<\frac{3}{4}z$ for $|z|>1$ and $|z|\leq 1$ respectively when $\sigma>0$, and $|z+\omega|< |z|-\log^{\frac{1}{8}}|z|$  and $|z+\omega|<\frac{3}{4}z$ for $|z|>1$ and $|z|\leq 1$ respectively when $\sigma=0$.

Indeed, if $z=\mu_{1}\neq 0$ is a pole of $f(z)$, then there is a sequence of poles, denoted by $\{\mu_{n}\}_{n=1}^{\infty}$ with $\lim_{n\rightarrow\infty}\mu_{n}=0$, of $f(z)$ by $f(z+\omega)\equiv f(z)$. Thus, $z=0$ is a non-isolated singularity of $f(z)$, which contradicts with $f(z)$ being meromorphic. Hence, the only possible pole of $f(z)$ is $z=0$. Then we write $f(z)=\frac{g(z)}{z^{m}}$, where $m$ is the multiplicity of pole at $z=0$,
$g(z)$ is an entire function and $g(0)\neq 0$. Since $f(z)$ is transcendental, we can find a finite complex number $\alpha$ such that the set $\{z:f(z)=\alpha\}$ is infinite. Thus, we can choose $0\neq \nu_{1}\in\{z:f(z)=\alpha\}$. By $f(z+\omega)\equiv f(z)$,
we have a sequence of $\alpha-$points, denoted by $\{\nu_{n}\}_{n=1}^{\infty}$ with $\lim_{n\rightarrow\infty}\nu_{n}=0$, of $f(z)$.
So $g(\nu_{n})=\nu_{n}^{m}f(\nu_{n})=\alpha\cdot\nu_{n}^{m}$. Note that $g(z)$ is entire, so we deduce that $g(0)=0$, which is a contradiction with $g(0)\neq 0$. Hence, $f(z)$ is an entire function.

Set $M=\{z:|z|\leq 10\}$, then $f(M)$ is bounded. For each $z\in \mathbb{C}\backslash M $, we can find a $z_{0}\in M$ such that $f(z)=f(z_{0})$. Thus, $f(z)$ is bounded, which implies that $f(z)$ is a constant. It contradicts with $f(z)$ being transcendental.

Combining the above two cases, we obtain that $f(z)$ is a constant.
\end{proof}
\medskip

\section{Preliminaries}\label{S:preliminaries}

\begin{lemma}[see \cite{He}, page 60]
\label{L:1}
Let $\alpha$ be a given constant with
$0<\alpha< 1$, then
\begin{equation}
\Big(\sum_{k=1}^{n}x_{k}\Big)^{\alpha}\leq \sum_{k=1}^{n} x_{k}^{\alpha},
\end{equation}
where $x_{k}\geq 0\ (k=1,2,\cdots,n)$.
\end{lemma}
~\\

\begin{lemma}[see \cite{He}, page 60]
\label{L:2}
Let $\varphi(x)$ be a positive-valued function on $[a,b]$. Then $\log \varphi(x)$ is integrable and
\begin{equation}
\frac{1}{b-a}\int_{a}^{b}\log \varphi(x)\ dx\leq \log \Big[\frac{1}{b-a}\int_{a}^{b}\varphi(x)\ dx\Big].
\end{equation}
\end{lemma}
~\\

\begin{lemma} [see \cite{He}, page 62]
\label{L:3}
Let $\alpha$, $0<\alpha < 1$ be given. Then
for every given complex number $\omega$, we have
\begin{equation}
\frac{1}{2\pi}\int_{0}^{2\pi}\frac{1}{|re^{i\theta}-\omega|^{\alpha}}d\theta\leq \frac{1}{(1-\alpha)r^{\alpha}}.
\end{equation}
\end{lemma}
~\\

\begin{lemma} [see \cite{He}, page 62]
\label{L:4}
Let $f(z)$ be a meromorphic function . Then
\begin{equation}
\left|\frac{f'(z)}{f(z)}\right|\leq \frac{8R}{(R-r)^{2}}\Big(T(R,f)+T\Big(R,\frac{1}{f}\Big)\Big)+\sum_{|a_{u}|< R}\frac{2}{|z-a_{u}|}+\sum_{|b_{v}|< R}\frac{2}{|z-b_{v}|},
\end{equation}
where $\{a_{u}\}$ and $\{b_{v}\}$ are the sets of zeros and poles of $f(z)$ in $D(0,R)$ respectively.
\end{lemma}
~\\

\begin{lemma} [see \cite{Mohon'ko}]
\label{L:5}
Let $f(z)$ ba a non-constant meromorphic function. Then for all irreducible rational functions in $f(z)$, \begin{equation}
R(f)=\frac{P(f)}{Q(f)}=\frac{\sum_{i=0}^{p}a_{i}(z)f^{i}(z)}{\sum_{j=0}^{q}b_{j}(z)f^{j}(z)},
\end{equation}
where $\{a_{i}(z)\}$ and $\{b_{j}(z)\}$ are small functions of $f(z)$, and $a_{p}(z)\not\equiv 0$, $b_{q}(z)\not\equiv 0$. Then,
\begin{equation}
T\left(r,R(f)\right)= \max \{p,q\}T(r,f)+S(r,f).
\end{equation}
Here, we say a meromorphic function $g(z)$ is a small function of $f(z)$ if $T(r,g)=o(T(r,f))$.
\end{lemma}
~\\

\begin{lemma}[see \cite{Chiang 1}]
\label{L:6}
Let $\alpha$ be a given constant with
$0<\alpha\leq 1$. Then there exists a constant $C_{\alpha}>0$
depending only on $\alpha$ such that
\begin{equation}
\log(1+x)\leq C_{\alpha}
x^{\alpha},
\end{equation} holds for $x\geq 0$. In particular, $C_{1}=1$.
\end{lemma}
~\\

\begin{lemma}[see \cite{Chiang 1}]
\label{L:7}
Let $\alpha$, $0<\alpha\leq 1$ be given and
$C_{\alpha}$ as given in Lemma \ref{L:6}. Then for any two complex numbers
$z_{1}$ and $z_{2}$ which do not vanish simultaneously, we have the inequality
\begin{equation}
\left|\log\left|\frac{z_{1}}{z_{2}}\right|\right|\leq C_{\alpha}\Big(\left|\frac{z_{1}-z_{2}}{z_{2}}\right|^{\alpha}+\left|\frac{z_{2}-z_{1}}{z_{1}}\right|^{\alpha}\Big).
\end{equation}
\end{lemma}
~\\

\begin{lemma}[see \cite{Gundersen}]
\label{L:8}
Let $z_{1}$, $z_{2}$, $\cdots$ be an infinite sequence of complex numbers that has no finite limit point, and that is ordered by increasing moduli. Let $n(t)$ denote the number of the points $\{z_{k}\}$ that lie in $|z|\leq t$. Let $\alpha>1$ be a given real constant. Then there exists a set $E\subset (1,\infty)$ that has finite logarithmic measure,
such that if $|z|\not\in E\cup [0,1]$, we have
\begin{equation}
\sum_{|z_{k}|\leq \alpha r}\frac{1}{|z-z_{k}|}<\alpha^{2}\frac{n\left(\alpha^{2}r\right)}{r}\log^{\alpha}r \log n\left(\alpha^{2}r\right),
\end{equation}
where $r=|z|$.
\end{lemma}

\section{Proof of Theorem \ref{T:limit-discrete-quotient}}
\begin{proof}
We distinguish two cases.
~\\
~\\
{\noindent\bf Case 1}. Suppose that $f(z)$ has no poles and zeros in $D(0,\, r+1):=\{z:|z|<r+1\}$.
Thus, we can choose $|\eta|\in \left(0,\frac{1}{2}\right)$ such that $z+\eta \in D(0,\, r+1)$ for all $z$ on $\{z:|z|=r\}$. It follows that $f(z+\eta)/f(z)$ is analytic on $\{z:|z|=r\}$.

Note that $\{z:|z|=r\}$ is a closed set, so $f(z+\eta)/f(z)$ is uniformly continuous on $\{z:|z|=r\}$.
Since $\lim\limits_{\eta\rightarrow
0}{f(z+\eta)}/{f(z)}=1$, so for arbitrary $ \varepsilon>0$, there exists $h_{1}(r,\varepsilon)>0$ such that
\begin{equation*}\left|\frac{f(z+\eta)}{f(z)}\right|<1+\varepsilon
\end{equation*}
whenever $|\eta|<h_{1}(r,\varepsilon)$, where $\lim\limits_{r\rightarrow\infty}h_{1}(r,\varepsilon)=0$.
Hence,
\begin{equation*}
m_{\eta}\Big(r,\frac{f(z+\eta)}{f(z)}\Big)< \log(1+\varepsilon)\leq \varepsilon
\end{equation*}
whenever $|\eta|<h_{1}(r,\varepsilon)$.

Therefore,
\begin{equation}\label{E:vanishing-period}
\lim\limits_{\eta\rightarrow
0}m_{\eta}\Big(r,\frac{f(z+\eta)}{f(z)}\Big)=0.
\end{equation}
Similarly, we have
\begin{equation*}
\lim\limits_{\eta\rightarrow
0}m_{\eta}\Big(r,\frac{f(z)}{f(z+\eta)}\Big)=0.
\end{equation*}

We also deduce
\begin{equation*}
\lim\limits_{r\rightarrow
\infty}m_{\eta}\Big(r,\frac{f(z+\eta)}{f(z)}\Big)\leq \varepsilon<2\varepsilon
\end{equation*}
since the $\eta\to 0$ and we can therefore apply the \eqref{E:vanishing-period}.
Hence
\begin{equation*}
\lim\limits_{r\rightarrow
\infty}m_{\eta}\Big(r,\frac{f(z+\eta)}{f(z)}\Big)=0\quad\mathrm{and}
\quad \lim\limits_{r\rightarrow
\infty}m_{\eta}\Big(r,\frac{f(z)}{f(z+\eta)}\Big)=0.
\end{equation*}

{\noindent\bf Case 2.} Suppose that $f(z)$ has poles and zeros in
$D(0,\, r+1)$. We define
\begin{equation*}
	F(z)=f(z)\frac{\prod_{v=1}^{N}(z-b_{v})}{\prod_{u=1}^{M}(z-a_{u})},
\end{equation*}
where $a_{u}\ (u=1,\,2,\,\cdots,M)$ and $b_{v}(v=1,\, 2,\, \cdots,\, N)$ to be the
zeros and poles of $f(z)$ in $D(0,\, r+1)$ respectively.
Thus, $F(z)$ does not have poles and zeros in $D(0,\, r+1)$. For all $z$ satisfying $|z|=r$, we can also choose $|\eta|\in \left(0,\frac{1}{2}\right)$ such that $z+\eta \in D(0,\, r+1)$.
Moreover,
\begin{equation*}
	F(z+\eta)=f(z+\eta)
\frac{\prod_{v=1}^{N}(z+\eta-b_{v})}{\prod_{u=1}^{M}(z+\eta-a_{u})}.
\end{equation*}
Then,
\begin{equation*}
\frac{f(z+\eta)}{f(z)}=\frac{F(z+\eta)}{F(z)}\prod_{u=1}^{M}\frac{z+\eta-a_{u}}{z-a_{u}}\prod_{v=1}^{N}\frac{z-b_{v}}{z+\eta-b_{v}}.
\end{equation*}

It follows from Lemma \ref{L:3} and Lemma \ref{L:7} with $0<\alpha<1$, that
\begin{equation*}
\begin{split}
&\ \ \ \left|\log\left|\frac{f(z+\eta)}{f(z)}\right|\right|
\leq \left|\log\left|\frac{F(z+\eta)}{F(z)}\right|\right|+\sum_{u=1}^{M}\left|\log\left|\frac{z+\eta-a_{u}}{z-a_{u}}\right|\right|
+\sum_{v=1}^{N}\left|\log\left|\frac{z-b_{v}}{z+\eta-b_{v}}\right|\right|\\
&\quad \leq \left|\log\left|\frac{F(z+\eta)}{F(z)}\right|\right|+C_{\alpha}|\eta|^{\alpha}\left[\sum_{u=1}^{M}\left(\frac{1}{|z-a_{u}|^{\alpha}}+\frac{1}{|z+\eta-a_{u}|^{\alpha}}\right)
\right.\\ &\qquad\left.
+\sum_{v=1}^{N}\left(\frac{1}{|z-b_{v}|^{\alpha}}+\frac{1}{|z+\eta-b_{v}|^{\alpha}}\right)\right].
\end{split}
\end{equation*}
Thus,
\begin{equation}
\begin{split}
\label{E:limit-quotient-1}
& \ \ \ \ m_{\eta}\Big(r,\frac{f(z+\eta)}{f(z)}\Big)+m_{\eta}\Big(r,\frac{f(z)}{f(z+\eta)}\Big)\leq m_{\eta}\Big(r,\frac{F(z+\eta)}{F(z)}\Big)+m_{\eta}\Big(r,\frac{F(z)}{F(z+\eta)}\Big)\\
&\qquad +C_{\alpha}|\eta|^{\alpha}\sum_{u=1}^{M}\left(\frac{1}{2\pi}\int_{0}^{2\pi}\frac{1}{|re^{i\theta}-a_{u}|^{\alpha}}d\theta
+\qquad \frac{1}{2\pi}\int_{0}^{2\pi}\frac{1}{|re^{i\theta}+\eta-a_{u}|^{\alpha}}d\theta\right)\\
&\qquad +C_{\alpha}|\eta|^{\alpha}\sum_{v=1}^{N}\left(\frac{1}{2\pi}\int_{0}^{2\pi}\frac{1}{|re^{i\theta}-b_{v}|^{\alpha}}d\theta
+\frac{1}{2\pi}\int_{0}^{2\pi}\frac{1}{|re^{i\theta}+\eta-b_{v}|^{\alpha}}d\theta\right)\\
&\qquad \leq m_{\eta}\Big(r,\frac{F(z+\eta)}{F(z)}\Big)+m_{\eta}\Big(r,\frac{F(z)}{F(z+\eta)}\Big)
+ \frac{2C_{\alpha}|\eta|^{\alpha}}{(1-\alpha)r^{\alpha}}(M+N)\\
&\qquad \leq m_{\eta}\Big(r,\frac{F(z+\eta)}{F(z)}\Big)+m_{\eta}\Big(r,\frac{F(z)}{F(z+\eta)}\Big)+ \frac{2C_{\alpha}|\eta|^{\alpha}}{(1-\alpha)r^{\alpha}}\left[n\left(r+1,f\right)
+n\Big(r+1,\frac{1}{f}\Big)\right]
\end{split}
\end{equation}
~\\
~\\
Since $F$ is free of zeros and poles in $|z|<r$, so we can apply Case 1 and the last inequality (\ref{E:limit-quotient-1})  to $F(z)$,  to deduce
\begin{equation*}
\lim\limits_{\eta\rightarrow
0}m_{\eta}\Big(r,\frac{f(z+\eta)}{f(z)}\Big)=0,\quad \mathrm{and}
\quad \lim\limits_{\eta\rightarrow
0}m_{\eta}\Big(r,\frac{f(z)}{f(z+\eta)}\Big)=0.
\end{equation*}
\medskip

On the other hand, we choose $\alpha=\frac{1}{2}$ in (\ref{E:limit-quotient-1}), and since $0<|\eta|< \alpha_{1}(r)$, where

$\alpha_{1}(r)=\min\big\{\log^{-\frac{1}{2}}r,1/\left(n(r+1)\right)^{2}\big\}$, $n(r)=n(r,\,f)+n\left(r,1/f\right)$, we have
\begin{equation*}
m_{\eta}\Big(r,\frac{f(z+\eta)}{f(z)}\Big)+m_{\eta}\Big(r,\frac{f(z)}{f(z+\eta)}\Big)
\leq m_{\eta}\Big(r,\frac{F(z+\eta)}{F(z)}\Big)+m_{\eta}\Big(r,\frac{F(z)}{F(z+\eta)}\Big)+ \frac{4C_{\frac{1}{2}}}{r^{\frac{1}{2}}},
\end{equation*}
which clearly tends to zero as $r\to \infty$. This proves (\ref{E:proximty-limit-0}).
\end{proof}
\medskip

\section{Proof of Corollary \ref{C:1}}
\begin{proof}
Since (\ref{E:proximty-limit-00}) holds for each positive real number $r$, hence given $\varepsilon>0$, then there exists $h(r,\, \varepsilon)>0$ such that
\begin{equation*}
m_{\eta}\Big(r,\frac{f(z+\eta)}{f(z)}\Big)< \varepsilon\quad\mathrm{and}
\quad m_{\eta}\Big(r,\frac{f(z)}{f(z+\eta)}\Big)< \varepsilon
\end{equation*}
\smallskip

\noindent whenever $|\eta|<h(r,\, \varepsilon)$. Note that
\begin{equation*}
\begin{split}
m(r,\, f(z)) &\leq m_{\eta}\Big(r,\frac{f(z)}{f(z+\eta)}\Big)+m_{\eta}(r,f(z+\eta))\\
&\leq m(r,f(z))+m_{\eta}\Big(r,\frac{f(z)}{f(z+\eta)}\Big)+m_{\eta}\Big(r,\frac{f(z+\eta)}{f(z)}\Big),
\end{split}
\end{equation*}
i.e.,
\begin{equation*}
-m_{\eta}\Big(r,\frac{f(z)}{f(z+\eta)}\Big)\leq m_{\eta}(r,f(z+\eta))-m(r,f(z))\leq m_{\eta}\Big(r,\frac{f(z+\eta)}{f(z)}\Big).
\end{equation*}
This implies
\begin{equation*}
\left|m_{\eta}(r,f(z+\eta))-m(r,f(z))\right|\leq m_{\eta}\Big(r,\frac{f(z)}{f(z+\eta)}\Big)+m_{\eta}\Big(r,\frac{f(z+\eta)}{f(z)}\Big)< 2\varepsilon
\end{equation*}
whenever $|\eta|<h(r,\varepsilon)$. Therefore
\begin{equation*}
\lim\limits_{\eta\rightarrow
0}m_{\eta}\left(r,f(z+\eta)\right)=m\left(r,f(z)\right).
\end{equation*}

\end{proof}

\section{Proof of Theorem \ref{T:limit-counting}}
\begin{proof}Let $\alpha_2(r)$ be defined in (\ref{E:alpha_2}). Since $0< |\eta|< \alpha_{2}(r)$,
then $n(0,f(z+\eta))=0$.

Applying the argument in \cite[(5.1), (5.2), (5.3) and (5.4)]{Chiang 1}, we deduce
\begin{equation}
\label{E:limit-counting-1}
\begin{split}
&|N_{\eta}(r,f(z+\eta))-N(r,f(z))|\\
&\leq |\eta|\bigg(\sum_{\substack{0<|b_{\mu}-\eta|<r,\\ b_{\mu}\neq 0}}\frac{1}{|b_{\mu}-\eta|}+\sum_{\substack{0<|b_{\mu}|<r,\\ b_{\mu}-\eta\neq 0}}\frac{1}{|b_{\mu}|}\bigg)
+n\left(0,f(z)\right)\log r.
\end{split}
\end{equation}

Note that $f(z)$ has no poles in $\overline{D}(0,\, h)\setminus \{0\}$, this implies that
	\begin{equation*}
\sum_{\substack{0<|b_{\mu}-\eta|\leq|\eta|,\\ b_{\mu}\neq 0}}\frac{1}{|b_{\mu}-\eta|}=0
\end{equation*}
under the assumption $0< |\eta|< \alpha_{2}(r)$. Thus,
\begin{equation}
\label{E:limit-counting-2}
\begin{split}\sum_{\substack{0<|b_{\mu}-\eta|<r,\\ b_{\mu}\neq 0}}\frac{1}{|b_{\mu}-\eta|}&=\sum_{\substack{0<|b_{\mu}-\eta|\leq|\eta|,\\ b_{\mu}\neq 0}}\frac{1}{|b_{\mu}-\eta|}+\sum_{|\eta|<|b_{\mu}-\eta|<r}\frac{1}{|b_{\mu}-\eta|}\\
&\leq\sum_{|\eta|<|b_{\mu}-\eta|<r}\frac{1}{|b_{\mu}|}\cdot\Big(1+\left|\frac{\eta}{b_{\mu}-\eta}\right|\Big)\\
&\leq 2\sum_{|\eta|<|b_{\mu}-\eta|<r}\frac{1}{|b_{\mu}|}
\leq 2\sum_{0<|b_{\mu}|<r+|\eta|}\frac{1}{|b_{\mu}|}.
\end{split}
\end{equation}
We deduce from (\ref{E:limit-counting-1}) and (\ref{E:limit-counting-2})
\begin{equation*}
|N_{\eta}\big(r,\, f(z+\eta)\big)-N(r,\, f(z))|\leq 3|\eta|\Big(\sum_{0<|b_{\mu}|<r+|\eta|}\frac{1}{|b_{\mu}|}\Big)+n\left(0,f(z)\right)\log r.
\end{equation*}
Since $0< |\eta|< \alpha_{2}(r)$, so
\begin{equation*}
|N_{\eta}(r,f(z+\eta))-N(r,f(z))|\leq n\left(0,f(z)\right)\log r+3.
\end{equation*}

Hence,
\begin{equation*}
N_{\eta}(r,f(z+\eta))=N(r,f(z))+\varepsilon_{1}(r),
\end{equation*}
where $|\varepsilon_{1}(r)|\leq n\left(0,f(z)\right)\log r+3$.

\end{proof}

\section{Proof of Theorem \ref{T:limit-characteristic}}
\begin{proof} It follows from the proofs of Theorem \ref{T:limit-discrete-quotient} and Theorem \ref{T:limit-counting} that for each $r>0$, there exists $\beta(r)>0$ such that whenever$0<|\eta|<\beta(r) :=\min\{\alpha_1(r),\, \alpha_2(r)\}$ (see the remark after the Theorem \ref{T:limit-characteristic}).

\begin{equation*}
\begin{split}
	 T_{\eta}\big(r,\, f(z+\eta)\big)&=m_{\eta}(r,\, f(z+\eta))+N_{\eta}(r,\, f(z+\eta))\\
&\leq m(r,\, f(z))+m_{\eta}\Big(r,\,\frac{f(z+\eta)}{f(z)}\Big)+N_{\eta}(r,\, f(z+\eta))\\
	& = T(r,\, f(z))+\varepsilon_{1}(r)+1
\end{split}
\end{equation*}
where $|\varepsilon_{1}(r)|\leq n\left(0,f(z)\right)\log r+3$. Similarly, we have
\begin{equation*}
	T(r,\, f(z))\leq T_{\eta}\big(r,\, f(z+\eta)\big)+\varepsilon_{1}(r)+1.
\end{equation*}
This proves (\ref{E:limit-characteristic}).
\end{proof}

\section{Proof of Theorem \ref{T:limit-discrete-quotient'}}
\begin{proof}
Since $f(z)$ is of finite order $\sigma$, then
\begin{equation*}
T(r,f)=O(r^{\sigma+\varepsilon}).
\end{equation*}
By choosing $R=2r$, $R'=3r$ and $\alpha=1-\varepsilon$ in \cite[Theorem 2.4]{Chiang 1}, we have
\begin{equation*}
 m\Big(r,\,\frac{f(z+\omega)}{f(z)}\Big)+ m\Big(r,\,\frac{f(z)}{f(z+\omega)}\Big)=O\big(r^{\sigma-(1-\beta)(1-\varepsilon)+\varepsilon}\big).
\end{equation*}

\end{proof}

\section{proof of Theorem \ref{T:limit-counting'}}
\begin{proof}
If $\sigma$ is nonzero, then we again apply \cite[(5.1--4)]{Chiang 1} to obtain
\begin{equation*}
\begin{split}
&|N\big(r,\, f(z+\omega)\big)-N(r,\,f(z))|\\
&\leq\qquad
|\omega|\bigg(\sum_{\substack{0<|b_{\mu}-\omega|<r,\\ b_{\mu}\neq 0}}\frac{1}{|b_{\mu}-\omega|}+\sum_{\substack{0<|b_{\mu}|<r,\\ b_{\mu}-\omega\neq 0}}\frac{1}{|b_{\mu}|}\bigg)
+O\left(\log r\right).
\end{split}
\end{equation*}
Note that $|\omega|<r^{\beta}<r$ and $r>1$, then
\begin{equation*}\label{E:omega_split}	
	\sum_{\substack{0<|b_{\mu}-\omega|<r,\\ b_{\mu}\neq 0}}\frac{1}{|b_{\mu}-\omega|}
=\sum_{\substack{0<|b_{\mu}-\omega|\leq |\omega|,\\ b_{\mu}\neq 0}}\frac{1}{|b_{\mu}-\omega|}+\sum_{|\omega|<|b_{\mu}-\omega|<r}\frac{1}{|b_{\mu}-\omega|}.
\end{equation*}
 Lemma~\ref{L:8} implies, with $\alpha=2$, that
\begin{equation*}
\label{E:bounded pole}
\begin{split}
\sum_{\substack{0<|b_{\mu}-\omega| \leq |\omega|,\\ b_{\mu}\neq 0}}\frac{1}{|b_{\mu}-\omega|}
&\leq \sum_{|b_{\mu}|\leq 2|\omega|}\frac{1}{|\omega-b_{\mu}|}\leq 4\cdot \frac{n\left(4|\omega|\right)}{|\omega|}\cdot\log^{2}|\omega|\cdot \log n\left(4|\omega|\right)\\
&=O\left(|\omega|^{\sigma-1+\varepsilon}\cdot \log^{3}|\omega|\right)
\end{split}
\end{equation*}
when $|\omega|$ is sufficiently large and outside a set of finite logarithmic measure of $|\omega|$.

Since
\begin{equation*}\label{E:omega_split_2}
\begin{split}
	\sum_{|\omega|<|b_{\mu}-\omega|<r}\frac{1}{|b_{\mu}-\omega|}
 &\leq\sum_{|\omega|<|b_{\mu}-\eta|<r}\frac{1}{|b_{\mu}|}\cdot\Big(1+\left|\frac{\omega}{b_{\mu}-\omega}\right|\Big)\\
	&<\sum_{|\omega|<|b_{\mu}-\omega|<r}\frac{2}{|b_{\mu}|}
\leq \sum_{0<|b_{\mu}|<2r}\frac{2}{|b_{\mu}|},
\end{split}
\end{equation*}
thus,
\begin{equation}
\label{E:upper bounded estimation}
\begin{split}
|N(r,\, f(z+\omega))-N(r,f(z))|\leq 3|\omega|\Big(\sum_{0<|b_{\mu}|<2r}\frac{1}{|b_{\mu}|}\Big)+O\left(|\omega|^{\sigma+\varepsilon}\cdot \log^{3}|\omega|\right)+O(\log r).
\end{split}
\end{equation}
But standard argument (see e.g. \cite[(5.9)--(5.12)]{Chiang 1} implies that
\begin{equation*}
\sum_{0<|b_{\mu}|<2r}\frac{1}{|b_{\mu}|}=O\left(r^{\sigma-1+\varepsilon}\right)
\end{equation*}
when $\sigma\geq 1$ and
\begin{equation}\label{E:order-zero-sum}
\sum_{0<|b_{\mu}|<2r}\frac{1}{|b_{\mu}|}=O(1)
\end{equation}
when $\sigma<1$.

Thus when $\sigma\geq 1$, we choose
$0<\varepsilon<\min\left\{ (\sigma-1)(1-\beta)/\beta,\ 1-\beta \right\}$.
Hence,
~\\
	\begin{equation}\label{E:final}
		N\big(r,\, f(z+\omega)\big)=N(r,\,f)+O\big(r^{\beta(\sigma+\varepsilon)}\cdot\log^{3}r\big)+O\big(r^{\sigma-(1-\beta)+\varepsilon}\big)
+O\left(\log r\right)
	\end{equation}
outside a set of finite logarithmic measure of $|\omega|$ and hence of $|r|$. Since $\varepsilon<1-\beta$,  this together with \eqref{E:final}  give  (\ref{E:counting-eta-1}).

On the other hand, when $0<\sigma<1$, the \eqref{E:upper bounded estimation} becomes
\begin{equation}\label{E:final_2}
		N\big(r,\, f(z+\omega)\big)=N(r,\,f)+O\big(r^{\beta}\big)+O\big(r^{\beta(\sigma+\varepsilon)}\cdot\log^{3}r\big).
+O\left(\log r\right)
	\end{equation}
we choose $0<\varepsilon< 1-\sigma$ in \eqref{E:final}. Since the term \eqref{E:order-zero-sum}
becomes bounded, the assumption on  $\varepsilon$ means that the term $O(r^\beta)$ is dominant over the term $r^{\beta(\sigma+\varepsilon)}$ and so \eqref{E:counting-eta-2} follows.



\item[(2)]
If $\sigma=\sigma(f)=0$,  then we choose $|\omega|<\log^{\frac{1}{2}}r<r$ for $r>1$. It follows similarly from (\ref{E:omega_split}) and the Lemma \ref{L:8}, with a different $\varepsilon$, $0<\varepsilon<1$, we have (\ref{E:bounded pole}) holds with $\sigma=0$.

It follows from (\ref{E:omega_split_2}) that \eqref{E:upper bounded estimation} holds with $\sigma=0$.

Note that $\sigma=\sigma(f)=0$, so that (\ref{E:order-zero-sum}) applies. Hence,
\begin{equation*}
N(r,f(z+\omega))=N(r,f)+O\big(\log^{\frac{\varepsilon}{2}}r\cdot \log^{3}\log r\big)+O\left(\log r\right)
=N(r,\, f)+ O(\log r)
\end{equation*}
outside a set of finite logarithmic measure.
\end{proof}
\bigskip

\section{Applications of vanishing period}
\label{S:applications}

It is known that one can recover the classical Painlev\'e equations from the corresponding discrete Painlev\'e equations \cite{Ramani} taking suitable limits of specifically designated change of variables. See for example, \cite{Halburd 5} and \cite{GR2014}. We consider limits of different sort between certain discrete equations and their continuous counterparts below, by making use of what we have established in this paper.
\medskip

\begin{example}
If the difference equation
\begin{equation}
\label{E:applications-2}
f(z+\eta)-f(z)= R(f(z),z,\eta) =\frac{a_{0}(z,\eta)+a_{1}(z,\eta)f(z)+\cdots+a_{p}(z,\eta)f^{p}(z)}
{b_{0}(z,\eta)+b_{1}(z,\eta)f(z)+\cdots+b_{q}(z,\eta)f^{q}(z)}
\end{equation}
with rational coefficients $a_{i}(z,\eta)\ (i=1,\cdots,p)$,\ $b_{j}(z,\eta)\ (j=1,\cdots,q)$,
admits a transcendental meromorphic solution $f(z)$ which is independent of $\eta$, where $\eta$ is a nonzero parameter such that
$$
	\lim\limits_{\eta\rightarrow 0}\frac{R(f(z),z,\eta)}{\eta}=\widehat{R}(f(z),z),
$$
when taken as a formal limit, is a rational function of $f(z)$ with rational coefficients. Then, $q=0$ and $p\leq 2$. Moreover, (\ref{E:applications-2}) will be reduced into a Riccati differential equation of the form
$f'(z)=a(z)+b(z)f(z)+c(z)f^{2}(z)$ with rational coefficients.
\end{example}

\begin{proof}
It follows from (\ref{E:applications-2}) after division of $\eta$ on both sides and an application of Lemma \ref{L:5} that we have
\begin{equation*}
\begin{split}
&\max\{p,q\}T(r,f(z))= T(r, R(f(z),z,\eta) )+O( \log |\eta|)+S(r,f(z))\\
&= T(r,f(z+\eta)-f(z))+O( \log |\eta|)+S(r,f(z))\\
&\leq T(r,f(z+\eta))+T(r,f(z))+O( \log |\eta|)+S(r,f(z)).
\end{split}
\end{equation*}
We deduce from Theorem \ref{T:limit-characteristic} that
~\\
\begin{equation}
\label{E:pq}
\max\{p,q\}T(r,f(z))\leq 2T(r,f(z))+O( \log |\eta|)+\varepsilon(r)+S(r,f(z))
\end{equation}
holds.
~\\
We choose $|\eta|=\min\{\alpha_{1}(r),\alpha_{2}(r)\}/2$,  where $\alpha_{1}(r)$ and $\alpha_{2}(r)$ were defined in \eqref{E:alpha_1} and \eqref{E:alpha_2} respectively. Note that if $f$ has at most finitely many zeros, then we are done. If, however, $f$ has infinitely many zeros, then we have, for a suitably chosen $\delta>0$ that
	\begin{equation}
		\label{E:recipicol-2}
			\begin{split}
		\sum_{0<|b_{\mu}|<r+\frac{1}{2}}\frac{1}{|b_{\mu}|}
		&=\frac{n(r+1/2)}{r+1/2}+\int_\delta^{r+\frac12} \frac{n(t)}{t^2}\,dt
		+O(1)\\
		&\ge \frac{n(r+1/2)}{r+1/2}+O(1)+\int_{\frac{r}{2}+\frac14}^{r+\frac12} \frac{n(t)}{t^2}\,dt
		\\
		&\ge \frac{n(r+1/2)}{r+1/2}+O(1)+\Big(r+\frac12-\frac{r}{2}-\frac14\Big)
		\frac{n(\frac{r}{2}+\frac14)}{(r+\frac12)^2}\\
		&=\frac{n(r+1/2)}{r+1/2}+O(1)+\Big(\frac12+o(1)\Big)\frac{n(\frac{r}{2}+\frac14)}{r}
		\end{split}
	\end{equation}
from which and \eqref{E:pq} we deduce $\max\{p,q\}\leq 2$.


On the other hand, note that (\ref{E:applications-2}) can be written as
\begin{equation*}
\frac{f(z+\eta)-f(z)}{\eta}=\frac{R(f(z),z,\eta)}{\eta}.
\end{equation*}
~\\
Letting $\eta\rightarrow 0$ as a formal limit, we obtain
\begin{equation}
\label{E:applications-3}
f'(z)=\widehat{R}(f(z),z),
\end{equation}
which is an equation considered by Malmquist.  Since this equation admits a meromorphic solution under our assumption, Malmquist's theorem (see \cite[p. 193]{Laine}) implies that the equation (\ref{E:applications-3}) reduces to a Riccati differential equation of the form
\begin{equation*}
	f^\prime(z)=a(z)+b(z)f(z)+c(z)f^{2}(z)
\end{equation*}
with rational coefficients.
Therefore, $q=0$ and $p\leq 2$.
\end{proof}

\begin{example}
If the difference equation
\begin{equation}
\label{E:applications-4}
\begin{split}
& f(z+\eta_{1}+\eta_{2})-f(z+\eta_{1})-f(z+\eta_{2})+f(z)= R(f(z),z,\eta_{1},\eta_{2}) \\&=\frac{a_{0}(z,\eta_{1},\eta_{2})+a_{1}(z,\eta_{1},\eta_{2})f(z)+\cdots+a_{p}(z,\eta_{1},\eta_{2})f^{p}(z)}
{b_{0}(z,\eta_{1},\eta_{2})+b_{1}(z,\eta_{1},\eta_{2})f(z)+\cdots+b_{q}(z,\eta_{1},\eta_{2})f^{q}(z)}
\end{split}
\end{equation}
with rational coefficients $a_{i}(z,\eta_{1},\eta_{2})\ (i=1,\cdots,p)$,\ $b_{j}(z,\eta_{1},\eta_{2})\ (j=1,\cdots,q)$,
admits a transcendental meromorphic solution $f(z)$ which is independent of $\eta_{1}$ and $\eta_{2}$ such that both the
	
		\[
			\lim\limits_{\eta_{1}\rightarrow 0}\frac{R(f(z),z,\eta_{1},\eta_{2})}{\eta_{1}}=R_{1}(f(z),z,\eta_{2})
		\]
		and the
	
		\[
		\lim\limits_{\eta_{2}\rightarrow 0}\lim\limits_{\eta_{1}\rightarrow 0}=\frac{R(f(z),z,\eta_{1},\eta_{2})}{\eta_{1}\eta_{2}}=\lim\limits_{\eta_{2}\rightarrow 0}\frac{R_{1}(f(z),z,\eta_{2})}{\eta_{2}}=R_{2}(f(z),z)
	\]
	are rational functions of $f(z)$ with rational coefficients. Then either (\ref{E:applications-4}) will be reduced into Painlev\'{e} equations $(I)$ or $(II)$ after taking limits, which implies that $q=0$ and $p\leq 3$, or (\ref{E:applications-4}) will be transformed into a reducible second order differential equation or that without Painlev\'{e} property.
\end{example}

\begin{proof}
We deduce from  (\ref{E:applications-4}) and Lemma \ref{L:5} that
\begin{equation*}
\begin{split}
&\max\{p,q\}T(r,f(z))=T(r,R(f(z),z,\eta_{1},\eta_{2}))+O( \log |\eta_{1}|)
+O( \log |\eta_{2}|)
+S(r,f(z))\\
&= T(r,f(z+\eta_{1}+\eta_{2})-f(z+\eta_{1})-f(z+\eta_{2})+f(z))+\\
&+ O( \log |\eta_{1}|)
+O( \log |\eta_{2}|)+S(r,f(z))\\
&\leq T_{\eta_{1}+ \eta_{2}}(r,f(z+\eta_{1}+\eta_{2}))+T_{\eta_{1}}(r,f(z+\eta_{1}))+T_{\eta_{2}}(r,f(z+\eta_{2}))+T(r,f(z))\\
&+O( \log |\eta_{1}|)
+O( \log |\eta_{2}|)+S(r,f(z)).
\end{split}
\end{equation*}

We  deduce from Theorem \ref{T:limit-characteristic} that
\begin{equation}
\label{E:pq-2}
\max\{p,q\}T(r,f(z))\leq 4T(r,f(z))+O( \log |\eta_{1}|)
+O( \log |\eta_{2}|)+3\varepsilon(r)+S(r,f(z)).
\end{equation}
~\\
We choose $|\eta_{1}|=|\eta_{2}|=\min\{\alpha_{1}(r),\alpha_{2}(r)\}/2$,  where $\alpha_{1}(r)$ and $\alpha_{2}(r)$ were defined in \eqref{E:alpha_1} and \eqref{E:alpha_2} respectively. According to 
\eqref{E:recipicol-2} Theorem \ref{T:limit-characteristic} applies, from which and \eqref{E:pq-2}  we deduce $\max\{p,q\}\leq 4$.
~\\
On the other hand, noting that (\ref{E:applications-4}) can be written as
\begin{equation*}
\frac{f(z+\eta_{2}+\eta_{1})-f(z+\eta_{2})}{\eta_{1}}-\frac{f(z+\eta_{1})-f(z)}{\eta_{1}}
=\frac{R(f(z),z,\eta_{1},\eta_{2})}{\eta_{1}}.
\end{equation*}
Letting $\eta_{1}\rightarrow 0$ as a formal limit, we get
\begin{equation*}
\frac{f'(z+\eta_{2})-f'(z)}{\eta_{2}}=\frac{1}{\eta_{2}}\cdot R_{1}(f(z),z,\eta_{2}).
\end{equation*}
Letting $\eta_{2}\rightarrow 0$ as a formal limit, we have
\begin{equation}
\label{E:applications-5}
f''(z)=R_{2}(f(z),z).
\end{equation}
~\\
Then, see for example  \cite[\S 14.4]{Ince}, either (\ref{E:applications-5}) is a reducible second-order differential equation, which can be solved by known special functions, or that without Painlev\'{e} property or it is $P_{I}: f''(z)=6 f^{2}(z)+z$ or $P_{II}:f''(z)=2 f^{3}(z)+zf(z)+\alpha$, where $\alpha$ is a constant.
~\\
Under the second case, we  deduce that $q=0$ and $p\leq 3$.
\end{proof}

The following theorem is a limiting analogue of Clunie lemma. Although the basic idea goes back to that of Clunie \cite{Laine}, we apply our Theorem \ref{T:limit-discrete-quotient} instead of the logarithmic derivative estimate \cite{Laine} and the logarithmic difference estimate \cite{Chiang 1}.
~\\
\begin{theorem}
Let $f(z)$ be a nonconstant meromorphic solution of
\begin{equation*}
f^{n}(z)P(z,f)=Q(z,f),
\end{equation*} where $P(z,f)$ and $Q(z,f)$ are difference polynomials in $f(z)$ and its steps of shifts are nonzero parameters. If the degree of $Q(z,f)$ is at most $n$, then
\begin{equation*}
\lim\limits_{\Gamma\rightarrow
0} m\left(r,P(z,f)\right)=o(T(r,f)),
\end{equation*} where $\Gamma$ is the set of all these steps in $P(z,f)$. Here, $\Gamma\rightarrow 0$ means the maximum length of the steps tends to zero.
\end{theorem}

\section{A re-formulation of logarithmic derivative lemma}
\label{S:reformulation}
We give an alternative derivation of Nevanlinna's original logarithmic derivative lemma $m(r,\, f'/f)=O(\log r)$
via a formal limiting process of a new difference-type estimate of
$$m\bigg(r,\,  \frac{1}{\eta}\Big(\frac{f(z+\eta)}{f(z)}-1\Big)\bigg)\longrightarrow m\Big(r,\, \frac{f^\prime}{f}\Big),\qquad \eta\rightarrow 0.
$$
\medskip

\begin{theorem}
\label{T:reformulation-logarithmic derivative}
Let $f(z)$ be a meromorphic
function, $r$,\ $R$ and $R'$ be positive real numbers satisfying
$0<r<R<R'$, and $0<\alpha<1$ be a constant. Then
\begin{equation}\label{E:reformulation-logarithmic derivative}
\begin{split}
&\ \ \ \ \lim\limits_{\eta\rightarrow
0}m_\eta\left(r,\frac{1}{\eta}\left(\frac{f(z+\eta)}{f(z)}-1\right)\right)\\
&\leq \frac{1}{\alpha}\log \left(1+\frac{8R^{\alpha}}{(R-r)^{2\alpha}}\left(T^{\alpha}(R,f)+T^{\alpha}\left(R,\frac{1}{f}\right)\right)
+\frac{3\left(N(R',f)+N\left(R',\frac{1}{f}\right)\right)}{(1-\alpha)r^{\alpha}\log
\frac{R'}{R}}\right)\\
&+\frac{1}{\alpha}\log \left(2^{\alpha}
+\frac{N(R',f)+N\left(R',\frac{1}{f}\right)}{(1-\alpha)r^{\alpha}\log
\frac{R'}{R}}\right)+2\log2.
\end{split}
\end{equation}
\end{theorem}
~\\

\begin{proof}
We distinguish two cases.
~\\
~\\
{\noindent\bf Case 1.} If $f(z)$ does not have zeros and poles in
$\overline{D}(0,\, R)$, then $f(z)$ is analytic on $D(0,R):=\{z:|z|<R\}$.
Thus, for all $z$ satisfying $|z|=r<R$, we can choose $|\eta|(>0)$ sufficiently small such that $z+\eta \in D(0,R)$ and
\begin{equation*}
f(z+\eta)-f(z)=\eta f'(z)+o(\eta),
\end{equation*}
as $\eta\to 0$.
So
\begin{equation*}
\frac{1}{\eta}\Big(\frac{f(z+\eta)}{f(z)}-1\Big)=
\frac{f'(z)}{f(z)}+\frac{o(\eta)}{\eta}\frac{1}{f(z)}.
\end{equation*}
~\\
Since $f(z)$ does not have zeros in $\overline{D}(0,R)$,
then $1/f(z)$ is analytic on
$D(0,R)$ and continuous on $\overline{D}(0,R)$.
By the Maximum Modulus Principle, there exists $M_{1}>0$ such that
$|1/f(z)|<M_{1}$ for all $z\in D(0,R)$.

We further choose
$|\eta|(>0)$ sufficiently small such that
$|\frac{o(\eta)}{\eta}\frac{1}{f(z)}|< \frac{1}{2}$ for all $z\in D(0,R)$.
Hence, when $|\eta|(>0)$ is sufficiently small, we have
\begin{equation}
\label{E:reformulation-logarithmic derivative-0}
\begin{split}
m_{\eta}\Big(r,\frac{1}{\eta}\Big(\frac{f(z+\eta)}{f(z)}-1\Big)\Big)&\leq m\Big(r,
\frac{f'(z)}{f(z)}\Big)+m\Big(r,\frac{o(\eta)}{\eta}\frac{1}{f(z)}\Big)+\log
2\\
&\leq  m\Big(r,
\frac{f'(z)}{f(z)}\Big)+\log 2.
\end{split}
\end{equation}
~\\
{\noindent\bf Case 2.} If $f(z)$ has zeros and poles in
$\overline{D}(0,R)$, then we define
\begin{equation}
\label{E:reformulation-logarithmic derivative-1}
F(z)=f(z)
\frac{\prod_{v=1}^{N}(z-b_{v})}{\prod_{u=1}^{M}(z-a_{u})},
\end{equation}
where $a_{u}(u=1,2,\cdots,M)$ and $b_{v}(v=1,2,\cdots,N)$ are the
zeros and poles of $f(z)$ on $\overline{D}(0,R)$ respectively. Thus, $F(z)$ is free of poles and zeros on $\overline{D}(0,R)$.

For all $z$ satisfying $|z|=r<R$, we can choose $|\eta|(>0)$ sufficiently small such that $z+\eta \in D(0,R)$. We deduce
\begin{equation}
\label{E:reformulation-logarithmic derivative-1'}
\begin{split}&\ \ \ \ \frac{1}{\eta}\Big(\frac{f(z+\eta)}{f(z)}-1\Big)\\
&=\frac{1}{\eta}\Big(\frac{F(z+\eta)}{F(z)}-1\Big)
+\frac{F(z+\eta)}{F(z)}\cdot\frac{1}{\eta}\Big(\prod_{u=1}^{M}\frac{z+\eta-a_{u}}{z-a_{u}}\prod_{v=1}^{N}\frac{z-b_{v}}{z+\eta-b_{v}}-1\Big).
\end{split}
\end{equation}
~\\
From (\ref{E:reformulation-logarithmic derivative-1}), we have
\begin{equation}
\label{E:reformulation-logarithmic derivative-2}
\log F(z)=\log f(z)+\sum_{v=1}^{N}\log (z-b_{v})-\sum_{u=1}^{M}\log (z-a_{u})+2k\pi i,
\end{equation}
for some $k\in \mathbb{Z}$. Taking logarithmic derivatives on both sides of (\ref{E:reformulation-logarithmic derivative-1}) and an application of Lemma \ref{L:4} allow us to deduce
~\\
\begin{equation}
\label{E:reformulation-logarithmic derivative-3}
\begin{split}
&\left|\frac{F'(z)}{F(z)}\right|\leq \left|\frac{f'(z)}{f(z)}\right|+\sum_{v=1}^{N}\frac{1}{|z-b_{v}|}+\sum_{u=1}^{M}\frac{1}{|z-a_{u}|}\\
&\leq \frac{8R}{(R-r)^{2}}\Big(T(R,f)+T\Big(R,\frac{1}{f}\Big)\Big)+\sum_{|a_{u}|< R}\frac{2}{|z-a_{u}|}+\sum_{|b_{v}|<R}\frac{2}{|z-b_{v}|}\\
&\qquad +\sum_{v=1}^{N}\frac{1}{|z-b_{v}|}+\sum_{u=1}^{M}\frac{1}{|z-a_{u}|}\\
&\leq\frac{8R}{(R-r)^{2}}\Big(T(R,f)+T\Big(R,\frac{1}{f}\Big)\Big)+\sum_{v=1}^{N}\frac{3}{|z-b_{v}|}+\sum_{u=1}^{M}\frac{3}{|z-a_{u}|}.
\end{split}
\end{equation}
~\\

Taking $\log^+$ of both sides of  (\ref{E:reformulation-logarithmic derivative-3}) and applying Lemma \ref{L:1},  with $0<\alpha<1$ yields
\begin{equation}\label{E:logF}
\begin{split}
&\ \ \ \ \log^{+}\left|\frac{F'(z)}{F(z)}\right|\leq \frac{1}{\alpha}\log \left(1+\left|\frac{F'(z)}{F(z)}\right|\right)^{\alpha}\\
&\leq \frac{1}{\alpha}\log \left(1+\frac{8R}{(R-r)^{2}}\Big(T(R,f)+T\Big(R,\frac{1}{f}\Big)\Big)+\sum_{v=1}^{N}\frac{3}{|z-b_{v}|}+\sum_{u=1}^{M}\frac{3}{|z-a_{u}|}\right)^{\alpha}\\
&\leq \frac{1}{\alpha}\log \left(1+\frac{8R^{\alpha}}{(R-r)^{2\alpha}}\Big(T^{\alpha}(R,f)+T^{\alpha}\Big(R,\frac{1}{f}\Big)\Big)+\sum_{v=1}^{N}\frac{3}{|z-b_{v}|^{\alpha}}+\sum_{u=1}^{M}\frac{3}{|z-a_{u}|^{\alpha}}\right).
\end{split}
\end{equation}
~\\
A straighforward application of Lemma \ref{L:2} and Lemma \ref{L:3} to \eqref{E:logF} give the estimate

\begin{equation}
\label{E:reformulation-logarithmic derivative-4}
m\Big(r,\frac{F'(z)}{F(z)}\Big) \leq \frac{1}{\alpha}\log \bigg(1+\frac{8R^{\alpha}}{(R-r)^{2\alpha}}\Big(T^{\alpha}(R,f)+T^{\alpha}\Big(R,\frac{1}{f}\Big)\Big)+\frac{3}{(1-\alpha)r^{\alpha}}(M+N)\bigg)
\end{equation}
~\\
Note that $0<R<R'$, we have
\begin{equation}
\label{E:reformulation-logarithmic derivative-5}
n(R,f)\leq \frac{N(R',f)}{\log \frac{R'}{R}},\quad\mathrm{and}\quad
n\Big(R,\frac{1}{f}\Big)\leq
\frac{N\Big(R',\frac{1}{f}\Big)}{\log
\frac{R'}{R}}.
\end{equation}
~\\
Combining (\ref{E:reformulation-logarithmic derivative-4}), (\ref{E:reformulation-logarithmic derivative-5}), we deduce that
\begin{equation}
\label{E:reformulation-logarithmic derivative-7}
\begin{split}
m\Big(r,\frac{F'(z)}{F(z)}\Big)\leq &\frac{1}{\alpha}\log \left(1+\frac{8R^{\alpha}}{(R-r)^{2\alpha}}\left(T^{\alpha}(R,f)+T^{\alpha}\Big(R,\frac{1}{f}\Big)\right)
\right.\\ &\left.+\frac{3\left(N(R',f)+N\Big(R',\frac{1}{f}\Big)\right)}{(1-\alpha)r^{\alpha}\log
\frac{R'}{R}}\right).
\end{split}
\end{equation}
~\\
On the other hand, an application of L'Hospital Rule yields
\begin{equation*}
\lim\limits_{\eta\rightarrow
0}\frac{1}{\eta}\left(\prod_{u=1}^{M}\frac{z+\eta-a_{u}}{z-a_{u}}\prod_{v=1}^{N}\frac{z-b_{v}}{z+\eta-b_{v}}-1\right)
=\sum_{u=1}^{M}\frac{1}{z-a_{u}}-\sum_{v=1}^{N}\frac{1}{z-b_{v}},
\end{equation*}
~\\
so that
\begin{equation*}
	\lim\limits_{\eta\rightarrow
0}\left|\frac{1}{\eta}\cdot\prod_{u=1}^{M}\frac{z+\eta-a_{u}}{z-a_{u}}\prod_{v=1}^{N}\frac{z-b_{v}}{z+\eta-b_{v}}-1\right|< \sum_{u=1}^{M}\frac{1}{|z-a_{u}|}+\sum_{v=1}^{N}\frac{1}{|z-b_{v}|}+1.
\end{equation*}
~\\
Hence, there exists $h>0$ such that
\begin{equation}\label{E:hospital-estimate}
\left|\frac{1}{\eta}\cdot\prod_{u=1}^{M}\frac{z+\eta-a_{u}}{z-a_{u}}\prod_{v=1}^{N}\frac{z-b_{v}}{z+\eta-b_{v}}-1\right|<
\sum_{u=1}^{M}\frac{1}{|z-a_{u}|}+\sum_{v=1}^{N}\frac{1}{|z-b_{v}|}+1
\end{equation}
whenever $0<|\eta|<h$. We apply Lemma \ref{L:1} to \eqref{E:hospital-estimate} deduce
\begin{equation}\label{E:reformulation-logarithmic derivative-6}
\begin{split}
&\log^{+}\left|\frac{1}{\eta}\cdot\prod_{u=1}^{M}\frac{z+\eta-a_{u}}{z-a_{u}}\prod_{v=1}^{N}\frac{z-b_{v}}{z+\eta-b_{v}}-1\right|
\leq \frac{1}{\alpha}\log \left(1+\left|\frac{1}{\eta}\cdot\prod_{u=1}^{M}\frac{z+\eta-a_{u}}{z-a_{u}}\prod_{v=1}^{N}\frac{z-b_{v}}{z+\eta-b_{v}}-1\right|\right)^{\alpha}\\
&\leq\frac{1}{\alpha}\log \left(2+\sum_{u=1}^{M}\frac{1}{|z-a_{u}|}+\sum_{v=1}^{N}\frac{1}{|z-b_{v}|}\right)^{\alpha}\leq \frac{1}{\alpha}\log \left(2^{\alpha}+\sum_{u=1}^{M}\frac{1}{|z-a_{u}|^{\alpha}}+\sum_{v=1}^{N}\frac{1}{|z-b_{v}|^{\alpha}}\right).
\end{split}
\end{equation}
~\\
Applying again Lemma \ref{L:2}, Lemma \ref{L:3}, (\ref{E:reformulation-logarithmic derivative-5}) to (\ref{E:reformulation-logarithmic derivative-6}) enable us to deduce
\begin{equation}
\label{E:reformulation-logarithmic derivative-8}
\begin{split}
&\ \ \ \ m_{\eta}\left(r,\frac{1}{\eta}\cdot\prod_{u=1}^{M}\frac{z+\eta-a_{u}}{z-a_{u}}\prod_{v=1}^{N}\frac{z-b_{v}}{z+\eta-b_{v}}-1\right)\\
	&\leq \frac{1}{\alpha}\log \left(2^{\alpha}
+\sum_{u=1}^{M}\frac{1}{2\pi}\int_{0}^{2\pi}\frac{1}{|re^{i\theta}-a_{u}|^{\alpha}}\ d\theta+\sum_{v=1}^{N}\frac{1}{2\pi}\int_{0}^{2\pi}\frac{1}{|re^{i\theta}-b_{v}|^{\alpha}}\ d\theta\right)\\
&\leq \frac{1}{\alpha}\log \left(2^{\alpha}+\frac{1}{(1-\alpha)r^{\alpha}}(M+N)\right)= \frac{1}{\alpha}\log \left(2^{\alpha}+\frac{1}{(1-\alpha)r^{\alpha}}\Big(n(R,f)+n\Big(R,\frac{1}{f}\Big)\Big)\right)\\
&\leq \frac{1}{\alpha}\log \bigg(2^{\alpha}
+\frac{N(R',f)+N\Big(R',\frac{1}{f}\Big)}{(1-\alpha)r^{\alpha}\log
\frac{R'}{R}}\bigg).
\end{split}
\end{equation}
~\\
Since $F(z)$ is free of zeros and poles in $D(0,\, R)$, so we can apply the Case 1 of Theorem \ref{T:limit-discrete-quotient} which asserts that
\begin{equation*}
\lim\limits_{\eta\rightarrow
0}m_{\eta}\Big(r,\frac{F(z+\eta)}{F(z)}\Big)=0.
\end{equation*}
~\\
In particular, we can apply the case 1 with (\ref{E:reformulation-logarithmic derivative-0})
above to $F(z)$ so that
	\begin{equation}
	\label{E:reformulation-logarithmic derivative-9}
		m\Big(r, \frac{1}{\eta}\Big(\frac{F(z+\eta)}{F(z)}-1\Big)\Big)\le
		m\Big(r,\, \frac{F^\prime(z)}{F(z)}\Big)+\log 2
	\end{equation}
when $|\eta|$ is sufficiently small.

Then, the inequality (\ref{E:reformulation-logarithmic derivative}) follows from
(\ref{E:reformulation-logarithmic derivative-1'}), (\ref{E:reformulation-logarithmic derivative-7}), (\ref{E:reformulation-logarithmic derivative-8}) and (\ref{E:reformulation-logarithmic derivative-9}).
\end{proof}
~\\

In particular, we have the following corollary for finite order meromorphic functions.
~\\
\begin{corollary}
\label{C:2}
Let $f(z)$ be a meromorphic function of
finite order $\sigma$, then
\begin{equation}
\lim\limits_{r\rightarrow
\infty}\lim\limits_{\eta\rightarrow
0}m_{\eta}\Big(r,\frac{1}{\eta}\Big(\frac{f(z+\eta)}{f(z)}-1\Big)\Big)=O(\log r), \ \ when \ \ \ \sigma \geq 1;
\end{equation}
\begin{equation}
\lim\limits_{r\rightarrow
\infty}\lim\limits_{\eta\rightarrow
0}m_{\eta}\Big(r,\frac{1}{\eta}\Big(\frac{f(z+\eta)}{f(z)}-1\Big)\Big)=O(1), \ \ when \ \ \ \sigma < 1.
\end{equation}
\end{corollary}
~\\
\begin{proof}
Since $f(z)$ is a non-constant meromorphic function of finite order $\sigma$,
then given $\varepsilon$ satisfying $0<\varepsilon<2$, we have
\begin{equation*}
T(r,\, f)\leq r^{\sigma+\frac{\varepsilon}{2}}
\end{equation*}
when $r$ is sufficiently large.
We choose $\alpha=1-\frac{\varepsilon}{2}$, $R=2r$ and $R'=3r$ in
Theorem \ref{T:reformulation-logarithmic derivative}, then the above limits follow.
\end{proof}
\bigskip

\section{Concluding remarks}  This paper has established a way that allows one to recover the classical little Picard theorem for meromorphic functions from the corresponding little Picard theorem for difference operators. One way to consider the original little Picard theorem is that it is a consequence of the meromorphic function belonging to the kernel of a differential operator.  Our formulations of Nevanlinna theories enable us to see that the meromorphic functions belonging to the respectively kernels
of vanishing and infinite periods varying steps difference operators  must reduce to constants. This allows us to treat the classical results as limiting cases of the discrete analogues. As the discrete-continuous interplay  has always been a new source of inspiration (see e.g. \cite{Ablowitz}, \cite{Ramani}), the current paper offers a concrete approach to achieve this interplay between discrete and continuous operators by ways of Nevanlinna theory.

\subsection*{Acknowledegment} We would like to acknowledge useful comments from  our colleague Dr. T. K. Lam during the initial stage of this research. The authors would also like to thank the referees for their constructive comments that led to better presentation of the paper.

\bibliographystyle{amsplain}

\end{document}